\documentclass[12pt]{amsart}
\usepackage{amscd,verbatim}
\usepackage{amssymb}
\usepackage[T1]{fontenc}
\usepackage[all]{xy}
\usepackage[
colorlinks,linkcolor=blue,citecolor=blue,urlcolor=red]{hyperref}
\usepackage{bbold}

\newcommand{\A}{\mathbf{A}}

\renewcommand{\P}{\mathbf{P}}

\newcommand{\Z}{\mathbb{Z}}
\newcommand{\sA}{\mathcal{A}}
\newcommand{\sB}{\mathcal{B}}
\newcommand{\sC}{\mathcal{C}}

\newcommand{\sO}{\mathcal{O}}

\newcommand{\sX}{\mathcal{X}}
\newcommand{\sY}{\mathcal{Y}}

\newcommand{\Mod}{\operatorname{Mod}\hbox{--}}

\newcommand{\Cor}{\operatorname{\mathbf{Cor}}}

\newcommand{\ul}[1]{{\underline{#1}}}

\newcommand{\PST}{{\operatorname{\mathbf{PST}}}}

\newcommand{\DM}{\operatorname{\mathbf{DM}}}

\newcommand{\MDM}{\operatorname{\mathbf{MDM}}}
\newcommand{\MHSM}{\operatorname{\mathbf{MHSM}}}

\newcommand{\Spec}{\operatorname{Spec}}

\newcommand{\Sm}{\operatorname{\mathbf{Sm}}}
\newcommand{\Sch}{\operatorname{\mathbf{Sch}}}

\newcommand{\by}{\xrightarrow}

\newcommand{\iso}{\by{\sim}}

\newcommand{\ls}{{\operatorname{ls}}}
\newcommand{\pro}[1]{\text{\rm pro}_{#1}\text{\rm--}}
\newcommand{\proobj}[1]{\underset{#1}{\operatorname{``\varprojlim''}}}
\newcommand{\tr}{{\operatorname{tr}}}
\newcommand{\man}{{\operatorname{man}}}

\newcommand{\fin}{{\operatorname{fin}}}
\renewcommand{\o}{{\circ}}
\newcommand{\op}{{\operatorname{op}}}

\newcommand{\red}{{\operatorname{red}}}

\newcommand{\Zar}{{\operatorname{Zar}}}
\newcommand{\Nis}{{\operatorname{Nis}}}
\newcommand{\et}{{\operatorname{\acute{e}t}}}

\newcommand{\inj}{\hookrightarrow}

\newcommand{\id}{{\operatorname{Id}}}

\newcommand{\CH}{{\operatorname{CH}}}

\renewcommand{\lim}{\operatornamewithlimits{\varprojlim}}

\newcommand{\ol}{\overline}

\renewcommand{\phi}{\varphi}
\renewcommand{\epsilon}{\varepsilon}

\newcommand{\MCor}{\operatorname{\mathbf{MCor}}}

\newcommand{\Bl}{{\mathbf{Bl}}}

\newcommand{\M}{\mathbf{M}}

\newcommand{\ulM}{\underline{\M}}

\newcommand{\ulMPST}{\operatorname{\mathbf{\underline{M}PST}}}

\newcommand{\ulMCor}{\operatorname{\mathbf{\underline{M}Cor}}}
\newcommand{\ulNCor}{\operatorname{\mathbf{\underline{N}Cor}}}

\newcommand{\ulomega}{\underline{\omega}}

\newcommand{\olS}{\overline{S}}
\newcommand{\olT}{\overline{T}}

\newcommand{\olMCrv}{\operatorname{\mathbf{\overline{M}Crv}}}

\newcommand{\Comp}{\operatorname{\mathbf{Comp}}}
\newcommand{\Sep}{\operatorname{\mathbf{Sep}}}

\newcommand{\ulTCor}{\operatorname{\mathbf{\underline{T}Cor}}}
\newcommand{\TCor}{\operatorname{\mathbf{TCor}}}
\newcommand{\ulWCor}{\operatorname{\mathbf{\underline{W}Cor}}}
\newcommand{\WCor}{\operatorname{\mathbf{WCor}}}
\newcommand{\ulTphiCor}{\operatorname{\mathbf{\underline{T}}^\emptyset\mathbf{Cor}}}
\newcommand{\TphiCor}{\operatorname{\mathbf{T}^\emptyset\mathbf{Cor}}}
\newcommand{\TphiSm}{\operatorname{\mathbf{T}^\emptyset\mathbf{Sm}}}
\newcommand{\ulTbCor}{\operatorname{\mathbf{\underline{T}}^\mathrm{b}\mathbf{Cor}}}

\newcommand{\ulTredCor}{\operatorname{\mathbf{\underline{T}}^\mathrm{r}\mathbf{Cor}}}

\newcommand{\ulTvgCor}{\operatorname{\mathbf{\underline{T}}^\mathrm{vg}\mathbf{Cor}}}

\newcommand{\ulTvvgCor}{\operatorname{\mathbf{\underline{T}}^\mathrm{vvg}\mathbf{Cor}}}

\newcommand{\ulTsatCor}{\operatorname{\mathbf{\underline{T}}^\mathrm{sat}\mathbf{Cor}}}

\newcommand{\ulTPS}{{\operatorname{\mathbf{\ul{T}PS}}}}
\newcommand{\ulTPST}{{\operatorname{\mathbf{\ul{T}PST}}}}

\newcommand{\ulTredPST}{{\operatorname{\mathbf{\ul{T}}^{\mathrm{r}}\mathbf{PST}}}}

\newcommand{\TPST}{{\operatorname{\mathbf{TPST}}}}

\newcommand{\TSm}{{\operatorname{\mathbf{TSm}}}}
\newcommand{\olMSm}{{\operatorname{\mathbf{\ol{M}Sm}_{\log}}}}
\newcommand{\ulTSm}{{\operatorname{\mathbf{\ul{T}Sm}}}}

\newcommand{\IY}{\operatorname{\mathbf{IY}}}

\newcommand{\defset}[3]{\left\{{#1}\left\rvert \parbox{#2}{#3} \right. \right\}}

\def\bP{\mathbb{P}}

\newcounter{spec}
\newenvironment{thlist}{\begin{list}{\rm{(\roman{spec})}}%
{\usecounter{spec}\labelwidth=20pt\itemindent=0pt\labelsep=10pt}}%
{\end{list}}%

\newtheorem{lemma}{Lemma}[section]
\newtheorem{thm}[lemma]{Theorem}

\newtheorem{prop}[lemma]{Proposition}

\newtheorem{cor}[lemma]{Corollary}
\newtheorem{corollary}[lemma]{Corollary}

\theoremstyle{definition}
\newtheorem{defn}[lemma]{Definition}

\newtheorem{definition}[lemma]{Definition}

\theoremstyle{remark}

\newtheorem{rk}[lemma]{Remark}
\newtheorem{remark}[lemma]{Remark}
\newtheorem{remarks}[lemma]{Remarks}
\newtheorem{ex}[lemma]{Example}

\newtheorem{claim}[lemma]{Claim}

\numberwithin{equation}{section}

\begin{document}

\title{Modulus triples}
\author{Bruno Kahn}
\address{CNRS, Sorbonne Université and Université Paris Cité, IMJ-PRG\\ Case 247\\4 place
Jussieu\\75252 Paris Cedex 05\\France}
\email{bruno.kahn@imj-prg.fr}
\author{Hiroyasu Miyazaki}
\address{NTT Institute for Fundamental Mathematics, NTT Communication Science Laboratories, NTT Corporation, 3-1 Morinosatowakamiya, Atsugi, Kanagawa 243-0124 Japan}
\email{hiroyasu.miyazaki.ah@hco.ntt.co.jp}
\thanks{
The second author is supported by JSPS KAKENHI Grant (21K13783, 19K23413).
}
\date{\today}
\begin{abstract}
We develop a theory of modulus triples, for future motivic applications.
\end{abstract}










\maketitle

\hfill Preliminary version

\tableofcontents

\section*{Introduction}

One of the most important results in the theory of motives, as developed by Voevodsky in \cite{voetri}, is the comparison isomorphism between motivic cohomology groups and higher Chow groups. Recently, these two concepts have been generalized to the ``non $\mathbb{A}^1$-homotopy invariant'' world, in independent contexts. 
On the one hand, by F. Binda and S. Saito \cite{Binda-Saito}, higher Chow groups are generalized to ``higher Chow groups with modulus,'' denoted by $\CH^r (X,D,\ast)$; they can be considered as a cycle-theoretic counterpart of the relative $K$-groups $K_\ast (X,D)$. Here, $D$ is an effective Cartier divisor on a scheme $X$. We call $(X,D)$ a ``modulus pair.''
On the other hand, the theory of motives has been extended to a theory of ``motives with modulus'' \cite{motmod}. The motivation was to enlarge Voevodsky's category of motives $\DM$ into another triangulated category $\MDM$ which relates to a large class of non $\mathbb{A}^1$-invariant sheaves introduced in \cite{rec}, e.g. all commutative algebraic groups, de Rham-Witt complex, etc., see \cite{motmod2}.

It is natural to expect that Voevodsky's comparison theorem can be generalized to a comparison between the Hom groups of $\MDM$ and the higher Chow groups with modulus. However, this turns out to be hopeless. Indeed, motives with modulus and higher Chow groups with modulus have opposite functorialities. Motives with modulus are functorial with respect to a class of morphisms called ``admissible correspondences.'' A typical example of an admissible correspondence $(X,D) \to (X',D')$ is given by a morphism $f : X \to X'$ satisfying $D \geq f^\ast D'$. On the other hand, the (hypercohomology version of) the higher Chow groups with modulus should be functorial with respect to the \emph{coadmissible correspondences} $(X,D) \to (X',D')$ of Subsection \ref{s3.5}, a typical example of which is given by a morphism $f : X \to X'$ satisfying $D \leq f^\ast D'$. The functoriality expected above is true at least in this example, by the work of Wataru Kai \cite{kai}.  

Here is another example of this phenomenon. A key step in the construction of Voevodsky's category $\DM$ was to invert the first projections $X \times \mathbb{A}^1 \to X$ for all smooth schemes $X$ over the base field. In the construction of $\MDM$, this process was replaced with inverting all $(X,D) \otimes \ol{\square} \to (X,D)$, where $\ol{\square} = (\bP^1,\infty)$ and $\otimes$ is the tensor product of modulus pairs. 
As a consequence, the motives $M(X,D)$ associated to modulus pairs satisfy $M((X,D) \otimes \ol{\square}) \cong M(X,D)$. 
On the other hand, the higher Chow groups with modulus satisfy the  dual property $\CH^r((X,D),\ast) \cong \CH^r((X,D) \otimes \ol{\square}^\vee,\ast)$, where $\ol{\square}^\vee = (\bP^1 , -\infty)$. Note that the multiplicity of $\ol{\square}^\vee$ is \emph{negative}. In \cite{cubeinv}, the second author proved this isomorphism after generalizing the definition of the higher Chow groups with modulus to pairs $(X,D)$ of a scheme $X$ and a Cartier divisor $D$, which is not necessarily effective. This result suggests that motives with modulus and higher Chow groups with modulus may be connected to each other by ``changing the signs of divisors.''

The aim of this paper is to generalize the theory of motives with modulus so that the divisors $D$ admit a change of signs. One way for such a generalization might be to expand the theory of \cite{motmod} by replacing effective Cartier divisors with Cartier divisors: this can actually be done, see Lemma \ref{lA.1}. However, we make use of a more sophisticated idea, as in F. Binda's thesis \cite{binda}: consider two effective Cartier divisors. Thus the central objects that we study are triples
\[
T=(\ol{T},T^+,T^-),
\]
where $\ol{T}$ is a scheme and $T^+,T^-$ are \emph{effective} Cartier divisors on $\ol{T}$. We call them \emph{modulus triples}. Morally, considering such a triple corresponds to considering the pair $(\ol{T},T^+ - T^-)$ as suggested by the modulus condition in Definition \ref{d2.1}, and we do relate the two ideas in Proposition \ref{pA.1}. But the data of triples provide us with a much sharper and more flexible treatment of the theory. In the context of triples, ``change of signs'' is encoded as ``switching divisors.''


\bigskip

Here is a summary of the contents of the paper:

Section \ref{s2} contains the definition of the category of modulus triples under various incarnations which follow those of \cite{motmodI}. The main one is denoted by $\ulTCor$; like the category of modulus pairs $\ulMCor$, it carries a tensor structure and admits a monoidal forgetful functor to Voevodsky's category $\Cor$ of finite correspondences. It also contains $\ulMCor$ as the full subcategory given by the condition $T^-=0$. The modulus condition here is more sophisticated than in loc. cit., see Definition \ref{def-modcond} and Lemma \ref{general-mod}.

Section \ref{s3} is the core of this work: it studies a large number of subcategories of $\ulTCor$, depending on the relationship between the two divisors $T^+,T^-$ of a modulus triple and also on certain properties of morphisms. 
Its main point is to study to what extent the inclusion $\ulMCor\subset \ulTCor$ has a left or a right adjoint: such properties can be very useful for a future theory of (pre)sheaves. Clearly, $\ulMCor$ is contained in the full subcategory $\ulTphiCor$ of triples such that $T^+\cap T^-=\emptyset$, and both adjoints exist for this refined inclusion  by Theorem \ref{adjoint-phi}; moreover, the inclusion $\ulTphiCor\subset \ulTCor$ has a pro-left adjoint under resolution of singularities by Theorem \ref{pro-separation}, hence so does also the inclusion $\ulMCor\subset \ulTCor$ (Theorem \ref{adjoint-psi}). As for a right adjoint, it exists if we restrict to the non-full subcategory $\ulTvvgCor$ with the same objects as $\ulTCor$, but only morphisms in ``excellent position'' in the sense of Definition \ref{d2.3}: see Proposition \ref{p3.5}.

In Section \ref{s4}, we compare $\ulTCor$ with a quiver defined by F. Ivorra and T. Yamazaki in \cite{iy2} and to the `modulus data' of
F. Binda \cite{binda}: these comparisons work very well, which is for us an indication that the theory developed here is ``right''.

In the appendix, we introduce a larger version $\ulNCor$ of $\ulMCor$ involving not necessarily effective Cartier divisors, and show that the inclusion $\ulMCor\subset \ulTCor$ extends to a full embedding $\ulNCor\subset\ulTCor$ (Proposition \ref{pA.1}). This is a second indication for us that the present theory of modulus triples is the ``right'' one.


This 43 page paper only concerns categories of modulus triples. A theory of presheaves and sheaves will be developed in future work.

\subsection*{Acknowledgements} We thank Federico Binda, Florian Ivorra and Takao Yamazaki for a very enlightening discussion on Section \ref{s4}.

\subsection*{Conventions}\label{conv}
\begin{enumerate}
\item Throughout this paper, we fix a base field $k$. We denote by $\Sch$ the category of separated $k$-schemes of finite type, and by $\Sm$ its full category of smooth $k$-schemes.

\item Let $f : Y \to X$ be a morphism in $\Sch$, and let $A \subset X$ be a closed subscheme. Then, we abbreviate the fiber product $A \times_X Y$ by $f^{-1}(A)$.

\item For any scheme $X$, the \emph{normalization} of $X$ means the normalization of $X_\red$, and it is denoted by $X^N$.

\item Let $f : X \to S$ be a morphism of schemes, and let $\pi : \tilde{S} \to S$ be a blow up along some closed subscheme $Z \hookrightarrow S$. Consider the blow up $\tilde{X}:=\Bl_{f^{-1}(Z)} X \to X$ of $X$ along the closed subscheme $f^{-1}(Z) = Z \times_{S} X$. Then, the universal property of blowing up induces a unique morphism $\tilde{f} : \tilde{X} \to \tilde{S}$ which makes the resulting square diagram commute (see \cite[Corollary 7.15]{Hartshorne}). We call $\tilde{f}$ the \emph{strict transform of $f$ along $\pi$} (in \cite{Hartshorne}, this name is used only when $f$ is a closed immersion). Note that if $f$ is flat, then we have $\tilde{X} \simeq X \times_{S} \tilde{S}$.
\item Let $f:Y \to X$ be a morphism in $\Sch$, and let $D$ be a Cartier divisor on $X$. Suppose that the pullback $f^\ast D$  of a Cartier divisor is defined (in the sense of \cite[(21.4.2)]{EGA4}).
Then, we often write $D|_Y := f^\ast D$ for simplicity of notation. 
\end{enumerate}

\subsection*{A few facts on Cartier divisors} Let $f : Y \to X$ be a morphism in $\Sch$, and let $D \subset X$ be a Cartier divisor. 

\begin{enumerate}
\item If $D$ is effective and its pullback is defined, it corresponds to the closed subscheme $D \times_X Y$ of $Y$ \cite[(21.4.7)]{EGA4}. 
In particular, using the convention above, we then have $f^\ast (D) = f^{-1}(D)$.

\item Assume one of the following conditions:

\begin{itemize}
\item[(a)] $Y$ is flat, or
\item[(b)] $Y$ is reduced, and for any irreducible component $Y_i$ of $Y$, the image $f(Y_i )$ of $Y_i$ is not contained in the support $|D|$ of $D$.
\end{itemize}

Then, the pullback $f^\ast D$ is defined: for (a) see \cite[(21.4.5)]{EGA4} and for (b) see \cite[2.2.3]{sasa} (in summary: this is checked locally), or the Stacks Project, Lemma 30.13.13. (3)\footnote{\url{https://stacks.math.columbia.edu/tag/01WQ}} when $D$ is effective. 
\end{enumerate}

\section{The category of modulus triples}\label{s2}

\subsection{Definition of modulus triples}

\begin{definition}
A \emph{modulus triple} over $k$ is a triple $T = (\ol{T},T^+,T^-)$ which consists of 
\begin{itemize}
\item a scheme $\ol{T} \in \Sch$, and
\item two effective Cartier divisors $T^+$ and $T^-$ on $\ol{T}$.
\end{itemize}
We say that $(T^+,T^-)$ is a \emph{modulus structure} on $\ol{T}$.
A modulus triple $T$ is \emph{proper} if $\ol{T}$ is proper over $k$.
We set 
\begin{align*}
T^\o := \ol{T} \setminus T^+,
\end{align*}
and call it the \emph{interior} of $T$. We call $\ol{T}$ the \emph{total space} of $T$.
\end{definition}

\begin{definition}
A modulus triple $T$ is \emph{interiorly smooth} 
if $T^\o\in \Sm$.
A modulus triple $T$ is \emph{disjoint} if $T^+ \cap T^- = \emptyset$.
\end{definition}

\begin{remark}
For any modulus triple $T$ such that $T^\o$ is reduced (e.g., interiorly smooth modulus triples), the total space $\ol{T}$ is automatically reduced. This can be checked by the same argument as in \cite[Rem. 1.1.2 (3)]{motmodI}, but we provide a proof here for the sake of completeness. 
Since the problem is local, we may assume that $\ol{T}=\Spec (A)$ is affine and the effective Cartier divisor $T^+$ is principal. Let $a \in A$ be a non-zero divisor which generates the ideal of definition of $T^+$. Then, the natural morphism $A \to A[1/a]$ is injective. Since $T^\o = \ol{T} \setminus T^+ = \Spec(A[1/a])$ is reduced, the ring $A[1/a]$ is reduced. Therefore, $A$ is also reduced. This proves the assertion.
\end{remark}

\begin{definition}
Let $T = (\ol{T},T^+,T^-)$ be a modulus triple. We define a new modulus triple, denoted $T^\vee$, by
\[
T^\vee := (\ol{T},T^-,T^+).
\]
We have an evident equality $(T^\vee)^\vee = T$. Note that $(T^\vee)^\o\ne T^\o$ in general.
\end{definition}

\subsection{Modulus condition} Recall from \cite[Lemma A.1]{KM}:

\begin{lemma}\label{key-lem}
Let $X$ be a scheme.
Suppose given three effective Cartier divisors $D_1, D_2$ and $E$ on $X$ such that $E \leq D_i$ for each $i=1,2$.

Then, we have:
\[
\text{$E = D_1 \times_X D_2$ \ \ iff. \ \ $|D_1 - E| \cap |D_2 - E| = \emptyset $.}
\]
\end{lemma}

\begin{remark}
The ``inf'' of two effective Cartier divisors might be zero even if $|D_1| \cap |D_2| \neq \emptyset$: for example, consider the case $X = \mathbb{A}^2 = \Spec (k[x_1,x_2])$ and $D_i = \{x_i = 0\}$.
\end{remark}

\begin{proof}[Proof of Lemma \ref{key-lem}]
Regard the effective Cartier divisors $D_i - E$ as closed subschemes on $X$, and set
\[
Z := (D_1 - E) \times_X (D_2 - E).
\]
For a closed subscheme $i : V \to X$, we set \[I_V := \mathrm{Ker} (\mathcal{O}_X \to i_\ast \mathcal{O}_V ) . \]
Then, we have
\[
I_{D_1 \times_X D_2} = I_{D_1} + I_{D_2}.
\]
Since $I_{Z} = I_{D_1 - E} + I_{D_2 - E} = I_{D_1} \cdot I_E^{-1} + I_{D_2} \cdot I_E^{-1}$, we have
\[
I_{Z} \cdot I_E = (I_{D_1} \cdot I_E^{-1} + I_{D_2} \cdot I_E^{-1}) \cdot I_E = I_{D_1} + I_{D_2} ,
\]
where $I_{E}^{-1}$ denotes the inverse of the invertible ideal sheaf $I_E$.
Combining the above equalities, we obtain
\begin{equation}\label{eq-interpret}
I_{Z} \cdot I_E = I_{D_1 \times_X D_2} .
\end{equation}
Therefore, we have
\begin{align*}
|D_1 - E| \cap |D_2 - E| = \emptyset  &\Leftrightarrow Z = \emptyset \Leftrightarrow I_Z = \mathcal{O}_X \Leftrightarrow^\dag I_{D_1 \times_X D_2}  = I_E \\
&\Leftrightarrow  D_1 \times_X D_2 = E,
\end{align*}
where $\Leftrightarrow^\dag$ follows from (\ref{eq-interpret}) and the fact that $I_E$ is invertible.
This finishes the proof of Lemma \ref{key-lem}.
\end{proof}

\begin{definition}\label{d2.2}
Let $T = (\ol{T},T^+,T^-)$ be a modulus triple. Let $\pi_T : \Bl_{F_T}(\ol{T}) \to \ol{T}$ be the blow up of $\ol{T}$ along the closed subscheme 
\[F_T := T^+ \times_{\ol{T}} T^-, \] 
which we call the \emph{fundamental locus} of $T$. We denote by $E_T$ the exceptional divisor of $\pi$, i.e., $E_T : =\pi_T^{-1} (F_T)$. We set
\[
\tilde{T}^+ := \pi_T^\ast T^+ - E_T, \ \ \tilde{T}^- := \pi_T^\ast T^- - E_T.
\]
Note that $\tilde{T}^+$ and $\tilde{T}^-$ are effective Cartier divisors on $\Bl_{F_T}(\ol{T})$. 
We define a modulus triple $\Bl(T)$ by
\[
\Bl(T) := (\Bl_{F_T}(\ol{T}),\tilde{T}^+,\tilde{T}^-),
\]
and call it the \emph{separation of $T$}.
\end{definition}

\begin{remark}
If $T^+$ and $T^-$ intersect properly, then $\tilde{T}^-$ equals the strict transform of $T^-$.
\end{remark}

\begin{lemma}\label{l2.1}
For any modulus triple $T$, we have $\tilde{T}^+ \cap \tilde{T}^- = \emptyset$.  In other words, $\Bl(T)$ is disjoint. If $T$ is disjoint, $\Bl(T)=T$.
\end{lemma}

\begin{proof}
The first assertion follows from Lemma \ref{key-lem}.  The second one is obvious.
\end{proof}

\begin{definition}\label{def-modcond}
Let $T = (\ol{T},T^+,T^-)$ be a modulus triple.  
Let $f : X \to \ol{T}$ be a morphism in $\Sch$ which satisfies the following condition:
\begin{equation*}\label{eqF}
 f^{-1}(F_T) \text{ does not contain any irreducible component of } X.\tag{F}
\end{equation*}
 Let $\tilde{f} : \tilde{X} \to \Bl_{F_T}(\ol{T})$ be the strict transform of $f$ along $\pi_T : \Bl_{F_T}(\ol{T}) \to \ol{T}$.
We say that \emph{$f$ satisfies the modulus condition for $T$} if we have $\tilde{f}^{-1}(\tilde{T}^-) = \emptyset$.
If an immersion $f : V \to \ol{T}$ satisfies the modulus condition for $T$, then we simply say that $V$ satisfies the modulus condition for $T$.
\end{definition}

\begin{remark}\label{rem-componentwise}
(1) A morphism $X \to \ol{T}$ satisfies the modulus condition for $T$ if and only if for any irreducible component $V$ regarded as an integral scheme, the composite $V \to X \to \ol{T}$ satisfies the modulus condition for $T$. Indeed, denoting by $\tilde{V}$ the strict transform of $V$ along $\pi_T$, there exists a natural morphism $p:\sqcup_{V \in X^{(0)}} \tilde{V} \to \tilde{X}$. It suffices to prove that $p$ is surjective. Since $\tilde{X} \to X$ is birational, the surjectivity of the composite $\sqcup_{V \in X^{(0)}} \tilde{V} \to \sqcup_{V \in X^{(0)}} V \to X$ implies that $p$ is dominant. But it is also proper, since the composite $\sqcup_{V \in X^{(0)}} \tilde{V} \to \tilde{X} \to X$ equals the composite $\sqcup_{V \in X^{(0)}} \tilde{V} \to \sqcup_{V \in X^{(0)}} V \to X$, which is proper. 

(2) Note that Definition \ref{def-modcond} makes sense even if the pullbacks of effective Cartier divisors $f^\ast (T^+), f^\ast (T^-)$ are not defined. 
On the other hand, if they are defined, then the following Lemma \ref{general-mod} shows that the modulus condition in the sense of Definition \ref{def-modcond} can be regarded as a generalization of the classical modulus conditions (as in \cite{Krishna-Park}, \cite{Binda-Saito}, \cite{motmodI} etc.) to relative situations. 
\end{remark}


\begin{lemma}\label{general-mod}
Let $T= (\ol{T},T^+,T^-)$ be a modulus triple.  
Let $f : X \to \ol{T}$ be a morphism in $\Sch$ such that neither of $f^{-1}(T^+)$ nor $f^{-1}(T^-)$ contains any irreducible component of $X$ (In particular, $f$ satisfies Condition \eqref{eqF} of Definition \ref{def-modcond}).
Let $X^N \to X$ be the normalization, and let $\nu : X^N \to X \to \ol{T}$ be the composite. 
Then, the following conditions are equivalent:
\begin{itemize}
\item[(a)] The following inequality of effective Cartier divisors on $X^N$ holds:
\[\nu^\ast T^+  \geq \nu^\ast T^- .\]
\item[(b)] $f$ satisfies the modulus condition for $T$.
\end{itemize}
\end{lemma}


\begin{proof}
In view of Remark \ref{rem-componentwise} (1) and the fact that $X^N = \sqcup_{V \in X^{(0)}} V^N$, we may assume that $X=V$ is integral. 
Let $\tilde{f} : \tilde{V} \to \Bl_{F_T}(\ol{T})$ be the strict transform of $f$ along $\pi_T$, and let $\tilde{\nu} : \tilde{V}^N \to \tilde{V} \to \Bl_{F_T}(\ol{T})$ be the composite with the normalization morphism.  Note that the proper surjective morphism $\tilde{V}^N \to \tilde{V} \to V$ induces a proper surjective morphism $\tilde{V}^N \to V^N$ by the universality of normalization.
By \cite[Lemma 2.2]{Krishna-Park}, Condition (a) is then equivalent to the following inequality of effective Cartier divisors on $\tilde{V}^N$:
\[
\tilde{\nu}^\ast \pi_T^\ast T^+  \geq \tilde{\nu}^\ast \pi_T^\ast T^- .
\]

By adding $-\tilde{\nu}^\ast E_T$ to both sides, we get equivalently
\begin{equation}\label{eq-l1.2}
\tilde{\nu}^\ast \tilde T^+  \geq \tilde{\nu}^\ast \tilde T^-.
\end{equation}
By Lemma \ref{l2.1}, the inequality \eqref{eq-l1.2} is equivalent to $\tilde{\nu}^\ast \tilde T^- = \emptyset$,
hence, noting that $\tilde{V}^N \to \tilde{V}$ is surjective, to
Condition (b). 
\end{proof}

\begin{lemma}[Factorization lemma]\label{l-containment}
Let $T = (\ol{T},T^+,T^-)$ be a modulus triple, and let $f : X \to \ol{T}$ and $g : Y \to \ol{T}$ be morphisms in $\Sch$ satisfying Condition \eqref{eqF} of Definition \ref{def-modcond}. Assume that $g$ factors through $f$. Then, if $f$ satisfies the modulus condition for $T$, so does $g$. Moreover, if the morphism $Y \to X$ is proper and surjective, then the opposite implication also holds.
\end{lemma}

\begin{proof}
The first assertion is obvious since the strict transform of $g$ along $\pi_T$ factors through the strict transform of $f$ along $\pi_T$, by the universal property of blowing up.
The second assertion follows from the fact that the proper surjective morphism $Y \to X$ induces a proper surjective morphism $\tilde{Y} \to \tilde{X}$, where $\tilde{X}$ and $\tilde{Y}$ denote the strict transforms of $X$ and $Y$, respectively.
\end{proof}

The following lemma is quite useful.

\begin{lemma}\label{fully-contained2}
Let $T$ be a modulus triple. Let $f : X \to \ol{T}$ be a morphism in $\Sch$ which satisfies Condition \eqref{eqF} of Definition \ref{def-modcond}. If $f(X) \subset |T^+|$, then $f$ satisfies the modulus condition for $T$.
\end{lemma}


\begin{proof}
By Remark \ref{rem-componentwise} (1), we may assume that $X=V$ is integral. Let $\tilde{f} : \tilde{V} \to \Bl_{F_T}(\ol{T})$ be the strict transform of $f$ along $\pi_T$. We claim that $\tilde{f}(\tilde{V}) \subset |\tilde{T}^+|$. To see this, since $|\tilde{T}^+|$ is closed in $\Bl_{F_T}(\ol{T})$ and $\tilde{V}$ is irreducible, it suffices to prove that the generic point $\eta$ of $\tilde{V}$ maps into $\tilde{T}^+$. 
Since $\eta$ is outside the fiber of $F_T$ (hence of $E_T$) by Condition \eqref{eqF}, the evident equalities
\[
\pi_T^\ast T^+ |_{\tilde{T} \setminus E_T} = (\pi_T^\ast T^+ - E_T) |_{\tilde{T} \setminus E_T} = \tilde{T}^+ |_{\tilde{T} \setminus E_T}
\]
imply that $\eta \in \tilde{T}^+$, which proves the claim. Recalling that $\tilde{T}^+ \cap \tilde{T}^- = \emptyset$, we have $\tilde{V} \times_{\Bl_{F_T}(\ol{T})} \tilde{T}^- = \emptyset$. 
\end{proof}


\begin{lemma}[Modulus condition for proper image]\label{l-pc}
Let $T = (\ol{T},T^+,T^-)$ be a modulus triple, and let $f : \ol{U} \to \ol{T}$ be a proper morphism such that the pullback of effective Cartier divisors $U^+:=f^\ast T^+$, $U^-:=f^\ast T^-$ are defined. Let $U:=(\ol{U},U^+,U^-)$ be the induced modulus triple. Let $W$ be an integral closed subscheme of $\ol{U}$ and set $V:=f(W)_\red$.
Then, $V$ satisfies the modulus condition for $T$ if and only if $W$ satisfies the modulus condition for $U$.
\end{lemma}

\begin{proof}
Noting that $f^{-1}(F_T) = F_U$,  we have
\[\text{$V \not\subset F_T$ if and only if $W \not\subset F_U$. }\]
Assume that these equivalent conditions hold.

Let $\pi_T : \tilde{T} \to \ol{T}$ be the blow up of $\ol{T}$ along $F_T$, and let $\pi_U : \tilde{U} \to \ol{U}$ be the blow up of $\ol{U}$ along $F_U$. The universal property of blowing up implies that there exists a morphism $\tilde{f} : \tilde{U} \to \tilde{T}$ which makes the following diagram commute:
\[\xymatrix{
\tilde{U} \ar[r]^{\tilde{f}} \ar[d]_{\pi_U} & \tilde{T} \ar[d]^{\pi_T} \\
\ol{U} \ar[r]^f & \ol{T}.
}\]
By the assumption above, we can consider the strict transforms $\tilde{V} \subset \tilde{T}$ and $\tilde{W} \subset \tilde{U}$ of $V$ and $W$, respectively.

\begin{claim}\label{claim1.13} Set $\tilde{T}^- := \pi_T^\ast T^- -E_T$ and $\tilde{U}^- := \pi_U^\ast U^- -E_U$. Then, we have the following equality of effective Cartier divisors on $\tilde{U}$:
\[\tilde{f}^\ast \tilde{T}^- = \tilde{U}^-. \]
\end{claim}

\begin{proof}
The commutative diagram above implies that 
\begin{gather*}
\tilde{f}^\ast\pi_T^\ast T^-=\pi_U^\ast f^\ast T^-=\pi_U^\ast U^-, \\ 
\tilde{f}^\ast E_T=\tilde{f}^{-1}\pi_T^{-1}(F_T)=\pi_U^{-1} f^{-1}(F_T)=\pi_U^{-1} F_U=E_U,
\end{gather*}  
where the notation makes sense by Convention (4).

This finishes the proof of the claim.
\end{proof}

We come back to the proof of Lemma \ref{l-pc}. We need to prove
\[\text{$\tilde{V} \cap \tilde{T}^- = \emptyset$ if and only if $\tilde{W} \cap \tilde{U}^- = \emptyset$.}\]

The ``only if'' part is immediate from Claim \ref{claim1.13}. Assume that $\tilde{W} \cap \tilde{U}^- = \emptyset$. Then, we have
\[
\emptyset = \tilde{f}(\tilde{W} \cap \tilde{U}^-) = \tilde{f}(\tilde{W} \cap \tilde{f}^{-1}(\tilde{T}^-)) = \tilde{f}(\tilde{W}) \cap \tilde{T}^- 
\]
by the (set-theoretic) projection formula and Claim \ref{claim1.13}.  Since the proper surjective morphism $\tilde{f}$ induces a proper surjective morphism $\tilde{W} \to \tilde{V}$, we obtain $\tilde{V} \cap \tilde{T}^- = \emptyset$. 
\end{proof}

\subsection{The category $\protect\ulTCor$ of modulus triples}

\begin{definition}\label{d2.1}
Let $S,T$ be two modulus triples which are interiorly smooth.
An \emph{elementary modulus correspondence from $S$ to $T$} is an elementary finite correspondence $V\in \Cor (S^{\circ} , T^{\circ} )$ whose closure $\ol{V}$ in $\ol{S} \times \ol{T}$ satisfies the modulus condition for 
\[
S \otimes T^\vee = (\ol{S} \times \ol{T} ,  S^+\times \ol{T} + \ol{S}\times T^- , S^-\times \ol{T} + \ol{S}\times T^+),
\]
and is proper over $\ol{S}$. We write $\ulTCor(S,T)$ for the free abelian group with basis the elementary modulus correspondences from $S$ to $T$.
\end{definition}

\begin{remark}\label{rmq-position}
Let $V \in \Cor(S^\o,T^\o)$ be an elementary finite correspondence. Let $\ol{V}$ be the closure of $V$ in $\ol{S} \times \ol{T}$, that we regard as an integral $k$-scheme.  Then the pullbacks of effective Cartier divisors
\[
S^+ |_{\ol{V}}, \ \ S^- |_{\ol{V}}, \ \ T^+ |_{\ol{V}}
\]
are 
defined. Indeed, since $\ol{V}$ is dominant over $\ol{S}$, its image is not contained in $|S^\pm|$. Since $V \subset S^\o \times T^\o$, $\ol{V} \not\subset \ol{S} \times |T^+|$. 

We need to note that the pullback $T^- |_{\ol{V}}$ is not necessarily defined; equivalently, the pullback $(T^-\cap T^\o) |_V$ is not necessarily defined. However, in this situation, the modulus condition is always satisfied by the following lemma.
\end{remark}


\begin{lemma}\label{mod-criterion1}
Let $S$ and $T$ be modulus triples which are interiorly smooth. Let $V \in \Cor(S^\o,T^\o)$ be an elementary finite correspondence whose closure is proper over $\ol{S}$, and assume that the projection of $V$ in $T^\o$ is contained in $T^{-} \cap T^\o$ (or equivalently, in $T^{-}$). Then $V \in \ulTCor(S,T)$.
\end{lemma}

\begin{proof}
Let $\ol{V}$ be the closure of $V$ in $\ol{S} \times \ol{T} = \ol{S \otimes T^\vee}$.
First, we check that $\ol{V}$ is not contained in $F_{S \otimes T^\vee}$. Let $\eta$ be the generic point of $V$. Since $\eta \in S^\o \times T^\o$, It suffices to show 
\[
\eta \notin F_{S \otimes T^\vee}^\o := F_{S \otimes T^\vee} \cap S^\o \times T^\o .
\] 
The right hand side can be calculated as follows (as a set)
\[
 F_{S \otimes T^\vee}^\o = (\ol{S} \times T^-) \cap (S^- \times \ol{T}) \cap (S^\o \times T^\o).
\]
Therefore, if $\eta \in F_{S \otimes T^\vee}^\o$, then we have $\eta \in S^- \times \ol{T}$, which contradicts the fact that $\ol{V}$ is dominant over $\ol{S}$.
Therefore, we conclude that $\eta \notin F_{S \otimes T^\vee}$, as desired. 

By Lemma \ref{fully-contained2}, it now suffices to show that $\ol{V} \subset (S \otimes T^\vee )^+$. But
\[
\ol{V} \subset \ol{S} \times T^- \subset (S^+ \times \ol{T}+\ol{S} \times T^- ) = (S \otimes T^\vee )^+.
\]
\end{proof}

\begin{lemma}\label{-goes-}
Let $S,T$ be two (interiorly smooth) modulus triples and let $V\in \ulTCor(S, T)$ be elementary. Then
\[V \cap (|S^-| \times \olT)\subset \olS\times |T^-|.\]
If moreover the image of $V$ in $\ol{T}$ is not contained in $T^-$,  we have 
\[
T^- |_{V^N} \geq S^- |_{V^N} .
\]
\end{lemma}

\begin{proof}
If the image of $V$ in $\ol{T}$ is contained in $T^-$, the first assertion is trivial. Assume the contrary. Then, by Lemma \ref{general-mod}, the modulus condition implies that
\[
S^+ |_{\ol{V}^N} + T^- |_{\ol{V}^N} \geq S^- |_{\ol{V}^N} + T^+ |_{\ol{V}^N} ,
\]
where $\ol{V}^N$ is the normalization of $\ol{V}$. Therefore, denoting by $V^N$ the normalization of $V$, we have 
\[
T^- |_{V^N} \geq S^- |_{V^N} .
\]

This proves the second assertion. Since the natural morphism $V^N \to V$ is surjective, the first assertion follows. 
\end{proof}

\subsubsection{Composition}

\begin{prop}\label{p-composition}
Let $T_1,T_2,T_3$ be three modulus triples, and let $\alpha : T_1 \to T_2$ and $\beta : T_2 \to T_3$ be modulus correspondences. 
Note that we have $\alpha \in \Cor(T_1^{\circ},T_2^{\circ})$ and $\beta \in \Cor(T_2^{\circ},T_3^{\circ})$ by definition.
Then, the finite correspondence $\beta \circ \alpha \in \Cor(T_1^{\circ},T_3^{\circ})$ defines a modulus correspondence $T_1 \to T_3$.
\end{prop}

\begin{proof}
Clearly, we may assume that $\alpha$ and $\beta$ are elementary modulus correspondences.
Let $\gamma$ be an arbitrary irreducible component of $\beta \circ \alpha$.
It suffices to show that $\gamma$ defines a modulus correspondence $T_1 \to T_3$. 
Let $\ol{\alpha}$, $\ol{\beta}$ and $\ol{\gamma}$ be the closures in $\ol{T}_1 \times \ol{T}_2$, $\ol{T}_2 \times \ol{T}_3$ and $\ol{T}_1 \times \ol{T}_3$, respectively.

\textbf{Step 1.} First, we prove that $\ol{\gamma}$ is proper over $\ol{T}_1$.
Since the projection $T_1^{\circ} \times T_2^{\circ} \times T_3^{\circ} \to T_1^{\circ} \times T_3^{\circ}$ induces a surjection $|\alpha \times T_3^{\circ} \cap T_1^{\circ} \times \beta| \to |\beta \circ \alpha|$, there exists an irreducible component $\gamma'$ of $\alpha \times T_3^{\circ} \cap T_1^{\circ} \times \beta$ whose generic point lies over the generic point of $\gamma$.
Let $\ol{\gamma}'$ be the closure of $\gamma'$ in $\ol{T}_1 \times \ol{T}_2 \times \ol{T}_3$.
Since $\ol{\alpha}$ is proper over $\ol{T}_1$, $\ol{\gamma}'$ is also proper over $\ol{T}_1 \times \ol{T}_3$, hence there is a natural proper surjective morphism $\ol{\gamma}' \to \ol{\gamma}$. Moreover, since $\ol{\beta}$ is proper over $\ol{T}_2$, $\ol{\gamma}'$ is also proper over $\ol{\alpha}  \subset \ol{T}_1 \times \ol{T}_2$, therefore $\ol{\gamma}'$ is proper over $\ol{T}_1$. In particular, $\ol{\gamma}$ is proper over $\ol{T}_1$.


\textbf{Step 2.}  We prove that $\gamma \in \ulTCor(T_1,T_3)$. By \textbf{Step 1}, it suffices to prove that $\ol{\gamma}$ satisfies the modulus condition for $T_1 \otimes T_3^\vee$.
If $\gamma$ is contained in the fiber of $|T_3^-|$, then we are done by Lemma \ref{mod-criterion1}.

\begin{claim}\label{claim1.22}
Assume that $\gamma$ is not contained in the fiber of $|T_3^-|$. Then the following assertions hold:

(a) $\alpha$ is not contained in the fiber of $|T_2^-|$.

(b) $\beta$ is not contained in the fiber of $|T_3^-|$.

(c) The natural morphism $\gamma' \to \alpha$ is proper surjective, where $\gamma'$, as in \textbf{Step 1}, is an irreducible component of $\alpha \times T_3^{\circ} \cap T_1^{\circ} \times \beta$ which lies over $\gamma$. 
In particular, $\gamma'$ is not contained in the fiber of $|T_2^-|$. 
\end{claim}

\begin{proof}
(a): Assume that $\alpha \subset \ol{T}_1 \times T_2^-$. By Lemma \ref{-goes-}, we have $\beta \cap T_2^- \times \ol{T}_3 \subset \ol{T}_2 \times T_3^-$. Therefore, 
\begin{align*}
\alpha \times \olT_3 \cap \olT_1 \times \beta \subset (\ol{T}_1 \times T_2^- \times \olT_3) \cap  (\olT_1 \times \beta) \subset  \ol{T}_1 \times T_2^- \times T_3^-
\end{align*}
hence \[\ol{\alpha \times T^\o_3 \cap T^\o_1 \times \beta} \subset \ol{T}_1 \times \ol{T}_2 \times T_3^-.\]
This implies that $\ol{\gamma}$ is contained in $\ol{T}_1 \times T_3^-$, contrary to the assumption on $\gamma$.

(b): Assume that $\beta \subset \ol{T}_2 \times T_3^-$. Then, we have $\ol{\alpha \times T^\o_3 \cap T^\o_1 \times \beta} \subset  \ol{T}_2 \times \ol{T}_2 \times T_3^-$, therefore we obtain the same contradiction. 

(c): Note that the projection $T_1^\o \times T_2^\o \times T_3^\o \to T_1^\o \times T_2^\o$ induces a morphism $\gamma' \to \alpha$. Since $\alpha$ is finite over an irreducible component of $T_1^\o$, $\gamma'$ is proper over the same component of $T_1^\o$. Therefore, $\gamma' \to \alpha$ is proper. It suffices to prove that $\gamma' \to \alpha$ is dominant. Let $\eta$ be the generic point of $\gamma'$. Then it lies over the generic point of an irreducible component of $T_1^\o$ since $\gamma' \to \gamma$ is surjective. Let $\xi$ be the image of $\eta$ in $\alpha$. Note that $\xi$ lies over the generic point of an irreducible component of $T_1^\o$. Since $\alpha$ is irreducible and $\alpha \to T_1^\o$ is finite, $\xi$ must be the generic point of $\alpha$. This shows that $\gamma' \to \alpha$ is dominant, which proves the first assertion.
Finally, we prove the second assertion. If $\gamma' \subset \ol{T}_1 \times |T_2^-| \times \ol{T}_3$, then we have $\alpha  \subset  \ol{T}_1 \times |T_2^-|$ by the surjectivity of $\gamma' \to \alpha$, which is contrary to (a). 

This finishes the proof of Claim \ref{claim1.22}.
\end{proof}

In view of Claim \ref{claim1.22} and Remark \ref{rmq-position}, the pullbacks of effective Cartier divisors
\[
T_1^\pm |_{\ol{\alpha}}, \ T_2^\pm |_{\ol{\alpha}},\ T_2^\pm |_{\ol{\beta}}, \ T_3^\pm |_{\ol{\beta}},\ T_1^\pm |_{\ol{\gamma}}, \ T_3^\pm |_{\ol{\gamma}}, \  T_1^\pm |_{\ol{\gamma}'}, \ T_2^\pm |_{\ol{\gamma}'} \ T_3^\pm |_{\ol{\gamma}'}
\]
are all defined.

Since $\gamma' \to \alpha$ is surjective (in particular, dominant) by Claim \ref{claim1.22} (c), it induces a morphism $\ol{\gamma}^{\prime N} \to \ol{\alpha}^N$. Since $\alpha \in \ulTCor (T_1,T_2)$, we have
\begin{equation*}
T_1^+ |_{\ol{\alpha}^N} + T_2^- |_{\ol{\alpha}^N} \geq T_1^- |_{\ol{\alpha}^N} + T_2^+ |_{\ol{\alpha}^N},
\end{equation*}
which implies 
\begin{equation}\label{eq1.22a}
T_1^+ |_{\ol{\gamma}^{\prime N}} + T_2^- |_{\ol{\gamma}^{\prime N}} \geq T_1^- |_{\ol{\gamma}^{\prime N}} + T_2^+ |_{\ol{\gamma}^{\prime N}}.
\end{equation}
Note that all the pullbacks of effective Cartier divisors which appear in the above inequalities are defined. 

Let $\gamma_0$ be the image of $\gamma'$ by the projection $T_1^\o \times T_2^\o \times T_3^\o \to T_2^\o \times T_3^\o$.
Then, clearly we have $\gamma_0 \subset \beta$.
Let $\ol{\gamma}_0$ be the closure of $\gamma_0$ in $\ol{T}_2 \times \ol{T}_3$. Then, $\ol{\gamma}_0$ is not contained in the fibers of $|T_2|$ and of $|T_3|$, since the same holds for $\gamma'$ by Claim \ref{claim1.22}. Since $\beta$ satisfies the modulus condition for $T_2 \otimes T_3^\vee$, so does $\gamma_0$ by Lemma \ref{l-containment}. Therefore, Lemma \ref{general-mod} implies
\begin{equation*}
T_2^+ |_{\ol{\gamma}_0^N} + T_3^- |_{\ol{\gamma}_0^N} \geq T_2^- |_{\ol{\gamma}_0^N} + T_3^+ |_{\ol{\gamma}_0^N}.
\end{equation*}
Noting that we have a natural morphism $\ol{\gamma}^{\prime N} \to \ol{\gamma}_0^N$, we obtain 
\begin{equation}\label{eq1.22b}
T_2^+ |_{\ol{\gamma}^{\prime N}} + T_3^- |_{\ol{\gamma}^{\prime N}} \geq T_2^- |_{\ol{\gamma}^{\prime N}} + T_3^+ |_{\ol{\gamma}^{\prime N}}.
\end{equation}
Combining \eqref{eq1.22a} and \eqref{eq1.22b}, we get
\begin{equation}\label{eq1.22c}
T_1^+ |_{\ol{\gamma}^{\prime N}} + T_3^- |_{\ol{\gamma}^{\prime N}} \geq T_1^- |_{\ol{\gamma}^{\prime N}} + T_3^+ |_{\ol{\gamma}^{\prime N}}.
\end{equation}
Recall that we have shown in \textbf{Step 1} that the natural map $\ol{\gamma}' \to \ol{\gamma}$ is proper surjective. This induces a proper surjective morphism  $\ol{\gamma}^{\prime N} \to \ol{\gamma}^N$. Therefore, by \cite[Lemma 2.2]{Krishna-Park}, the inequality \eqref{eq1.22c} implies
\begin{equation*}
T_1^+ |_{\ol{\gamma}^{N}} + T_3^- |_{\ol{\gamma}^{N}} \geq T_1^- |_{\ol{\gamma}^{\prime N}} + T_3^+ |_{\ol{\gamma}^{N}}.
\end{equation*}
This shows that $\gamma \in \ulTCor(T_1,T_3)$, which finishes the proof of Proposition \ref{p-composition}.
\end{proof}

\begin{definition}
Write $\ulTCor$ for the category whose objects are interiorly smooth modulus triples over $k$ 
and morphisms are as in Definition \ref{d2.1}. The composition  in $\ulTCor$ is well-defined by Proposition \ref{p-composition}.
\end{definition}

There is an obvious forgetful, faithful functor
\begin{equation}\label{eq1.1}
\ulomega:\ulTCor\to \Cor 
\end{equation}
which sends $T$ to $T^\o$ and correspondences to themselves.

\begin{ex}\label{ex2.1} Let $T=(\ol{T},T^+,T^-)\in \ulTCor$, and let $D\subset \ol{T}$ be an effective Cartier divisor. There is a natural morphism in $\ulTCor$
\[(\ol{T},T^++D,T^-+D)\to (\ol{T},T^+,T^-)\]
which is an isomorphism if and only if $|D|\subseteq |T^+|$.
\end{ex}

\begin{ex}\label{ex2.2} Let $T\in \ulTCor$. Write $\ol{T}= \bigcup_i \ol{T}_i$, a union of irreducible components: thus $T^\o = \coprod_i T_i^\o$, where $T_i^\o=T^\o\cap \ol{T}_i$. For each $i$, define a modulus pair $T_i=(\ol{T}_i,T^+\mid_{\ol{T}_i},T^-\mid_{\ol{T}_i})$. Note that the pullback of divisors is defined since  $\ol{T}_i$ is not contained in $|T^+| \cup |T^-|$ for any $i$. Then the canonical map
\[\bigoplus_i T_i\to T,\]
induced by the finite morphism $\sqcup_i \ol{T}_i \to \ol{T}$, is an isomorphism in $\ulTCor$ (in fact, in the subcategory $\ulTCor^\fin$, which will be introduced in \S \ref{subsection-fin}).
\end{ex}

\begin{prop}\label{p2.1}
The category $\ulTCor$ has a symmetric monoidal structure such that the bifunctor
\[
-\otimes- : \ulTCor \times \ulTCor \to \ulTCor
\]
is defined on objects as follows.
For $S,T \in \ulTCor$, the tensor product $S \otimes T$ is given by
\[
S \otimes T := (\ol{S} \times \ol{T}, S^+ \times \ol{T}  + \ol{S} \times T^+ , S^- \times \ol{T}  + \ol{S} \times T^-).
\]
The unit object is given by
\[\mathbb{1} := (\Spec (k),\emptyset,\emptyset).\]
The functor $\ulomega$ of \eqref{eq1.1} is symmetric monoidal.
\end{prop}

\begin{proof} It is clear that $\ulomega(S\otimes T)= {}\ulomega(S)\times {}\ulomega(T)$. Let $S_1,S_2,T_1,T_2\in \ulTCor$ and  $\alpha\in \ulTCor(S_1,T_1)$ $\beta\in \ulTCor(S_2,T_2)$ be two modulus correspondences. We must show that $\alpha\otimes \beta\in \Cor(S_1^\o\times S_2^\o,T_1^\o\times T_2^\o)$ lies in $\ulTCor(S_1\otimes S_2,T_1\otimes T_2)$. Since
\[\alpha\otimes \beta = \alpha \otimes 1 \circ 1\otimes \beta\]
it suffices by symmetry to check that $S\otimes -$ is an endofunctor on $\ulTCor$ for any $S \in \ulTCor$. 

Let $\alpha : T_1 \to T_2$ be a morphism in $\ulTCor$. Since $\alpha \in \Cor(T_1^{\circ},T_2^{\circ})$, it induces a morphism $\beta := S^{\circ} \times \alpha \in \Cor(S^{\circ} \times T_1^{\circ},S^{\circ} \times T_2^{\circ})$. Noting that $(S \otimes T_i)^{\circ} = S^{\circ} \times T_i^{\circ}$, it suffices to check that $\beta \in \ulTCor(S \otimes T_1,S \otimes T_2)$. Without loss of generality, we may assume that $\alpha$ is an integral cycle. In this case, $\beta$ is also integral.
As a first case, we assume that the image of $\alpha$ in $T_2^\o$ is contained in $T_2^-$, which is equivalent to saying that the image of $\beta$ in $S^\o \times T_2^\o$ is contained in $\ol{S} \times T_2^-$. Then, in particular, the image of $\beta$ is contained in $(S\otimes T_2)^-$. Therefore, Lemma \ref{mod-criterion1} implies that $\alpha \in \ulTCor(S\otimes T_1,S\otimes T_2)$. 
Next, we assume that the image of $\alpha$ in $T_2^\o$ is not contained in $T_2^-$. Then, the image of $\beta$ in $S^\o \times T_2^\o$ is not contained in $\ol{S} \times T_2^-$. Therefore, letting $\beta^N$ be the normalization of the closure $\ol{\beta}$ of $\beta$ in $(\ol{S} \times \ol{T}_1) \times (\ol{S} \times \ol{T}_2)$,  all the pullbacks onto $\ol{\beta}$ of the effective Cartier divisors $S^\pm,T_1^{\pm},T_2^{\pm}$ are defined. By the modulus condition on $\alpha$, noting that there exists a dominant morphism $\ol{\beta}^N \to \ol{\alpha}^N$, we obtain
\[
T_1^+ |_{\ol{\beta}^N} + T_2^- |_{\ol{\beta}^N} \geq T_1^- |_{\ol{\beta}^N} + T_2^+ |_{\ol{\beta}^N}.
\]
By adding $S^+ |_{\ol{\beta}^N} +$ and $S^- |_{\ol{\beta}^N} +$ on both sides of the inequality, we have
\[
(S^+ |_{\ol{\beta}^N} + T_1^+ |_{\ol{\beta}^N}) + (S^- |_{\ol{\beta}^N} + T_2^- |_{\ol{\beta}^N}) \geq (S^- |_{\ol{\beta}^N} + T_1^- |_{\ol{\beta}^N}) + (S^+ |_{\ol{\beta}^N} + T_2^+ |_{\ol{\beta}^N}),
\]
which implies that $\beta \in \ulTCor(S\otimes T_1,S \otimes T_2)$. 

Finally, we leave it to the reader to define the associativity and commutativity constraints, and to check that they are given by modulus correspondences. 
\end{proof}




\subsection{The category $\protect\ulTSm$}

\begin{definition}
Let $\ulTSm$ be the subcategory of $\ulTCor$ with the same objects such that for any $S,T \in \ulTCor$ the set $\ulTSm(S,T)$ consists of those morphisms of schemes $S^\o \to T^\o$  whose graphs belong to $\ulTCor(S,T)$.
\end{definition}




\subsection{The categories $\protect\ulTCor^\fin$ and $\protect\ulTSm^\fin$}\label{subsection-fin}

\begin{definition}
Let $\ulTCor^\fin$ be the subcategory of $\ulTCor$ with the same objects such that for any $S,T \in \ulTCor$ we have
\[
\ulTCor^\fin(S,T) := \defset{\alpha \in \ulTCor(S,T)}{50mm}{For any component $V$ of $\alpha$, the closure of $V$ in $\ol{S} \times \ol{T}$ is finite over $\ol{T}$}.
\]

Moreover, we define $\ulTSm^\fin$ to be the subcategory of $\ulTCor$ with the same objects such that for any $S,T \in \ulTCor$ we have
\[
\ulTSm^\fin(S,T) := \defset{f \in \ulTSm(S,T)}{50mm}{$f$ extends to a morphism of schemes $\bar f:\ol{S} \to \ol{T}$.}.
\]

In this case, assuming $\bar S$ normal, the modulus condition may we written
\begin{equation}\label{eq2.1}
f(S-|S^+|)\subseteq |T^-| \text{ or } S^+ + \bar f^\ast T^- \geq S^- +\bar f^\ast T^+. 
\end{equation}

Note that for any $S,T \in \ulTCor$ with $\ol{S}$ normal, we have
\[
\ulTSm^\fin (S,T) = \ulTSm (S,T) \cap \ulTCor^\fin (S,T)
\]
by Zariski's connectedness theorem.
\end{definition}

\begin{rk} We can check that the finiteness condition is stable under composition in the same way as Step 1 of the proof of Proposition \ref{p-composition}. So $\ulTCor^\fin$ and $\ulTSm^\fin$ are indeed subcategories.
\end{rk}



\begin{definition}\label{def:dommin}
A morphism $f : S \to T$ in $\ulTSm$ is \emph{dominant} if any connected component of $S^\o$ is dominant over a connected component of $T^\o$.
A morphism $f : S \to T$ in $\ulTSm^\fin$ is \emph{minimal} if $f^\ast T^-$ is defined and we have $S^+ = f^\ast T^+$ and $S^- = f^\ast T^-$.
\end{definition}

\begin{remark}
(1) The composition of minimal morphisms  is again minimal by \cite[(21.4.4)]{EGA4}.

(2) If $f$ is a minimal morphism, then we have $S^+ - S^- = f^\ast T^+ - f^\ast T^-$. But the inverse implication does not hold in general.
\end{remark}

\subsection{The localization functor $b:\protect\ulTCor^{\fin} \to \protect\ulTCor$}\label{s2.6}

\begin{definition}
Define a class $\Sigma^\fin$ of morphisms in $\ulTSm^\fin$ by
\[
\Sigma^\fin := \defset{(f : S \to T) \in \ulTSm^\fin}{50mm}{$f$ is minimal, $f : \ol{S} \to \ol{T}$ is proper and $f^\o : S^\o \to T^\o$ is the identity.}.
\]
\end{definition}

\begin{prop}\label{localization1}\
\begin{itemize}
\item[(a)] The class $\Sigma_\fin$ enjoys a calculus of right fractions within $\ulTCor$ and $\ulTSm$.
\item[(b)] The inclusion functor $b:\ulTCor^\fin \to \ulTCor$ induces an equivalence of categories 
\[
\Sigma_\fin^{-1} \ulTCor^\fin \xrightarrow{\sim} \ulTCor.
\]
\item[(c)] The inclusion functor $\ulTSm^\fin \to \ulTSm$ induces an equivalence of categories 
\[
\Sigma_\fin^{-1} \ulTSm^\fin \xrightarrow{\sim} \ulTSm.
\]
\end{itemize}
\end{prop}

The proof is identical to that of \cite[Proposition 1.10.4]{motmod}, dealing with two moduluses instead of one, and we skip it.

\subsection{Squarable morphisms} Recall:

\begin{definition}[\protect{\cite[IV.1.4.0]{SGA3}}]
A morphism $f : S \to T$ in a category $\sC$ is \emph{squarable} if for any morphism $T' \to T$, the fiber product $S \times_{T} T'$ is representable in $\sC$.
\end{definition}

\begin{prop}\label{quarrable}
Any minimal morphism $f : S \to T$ in $\ulTSm^\fin$, with $f^\o : S^\o \to T^\o$ smooth, is squarable in $\ulTSm^\fin$.
\end{prop}

\begin{proof} It is in the same lines as that of \cite[Prop. 1.10.7]{motmodI}, but we have to elaborate a bit because of the two moduluses.

Let $g : T_1 \to T$ be a morphism in $\ulTSm^\fin$.  Consider the modulus triple 
\[
S_1 := (\ol{S} \times_{\ol{T}} \ol{T}_1 , \ol{S} \times_{\ol{T}} T^+_1 , \ol{S} \times_{\ol{T}} T^-_1  ).
\]
Then, the interior $S_1^\o$ equals $S^\o \times_{T^\o} T_1^\o$, which is smooth since $f^\o : S^\o \to T^\o$ is smooth. Indeed, we have
\[
S_1^\o = \ol{S} \times_{\ol{T}} T_1^\o = (\ol{S} \times_{\ol{T}} T^\o ) \times_{T^\o} T_1^\o = S^\o \times_{T^\o} T_1^\o,
\]
where the last equality holds since $S^+ = f^\ast T^+$ by the minimality of $f$. Therefore, we have $S_1 \in \ulTCor$. Moreover, we obtain the following commutative diagram in $\ulTSm^\fin$:
\[\xymatrix{
S_1 \ar[r]^{g_1} \ar[d]_{f_1} & S \ar[d]^f \\
T_1 \ar[r]^g & T,
}\]
where $f_1$ and $g_1$ are induced by the projections $\ol{S}_1 \to \ol{T}_1$ and $\ol{S}_1 \to \ol{S}$, respectively. Indeed, $f_1$ is minimal by definition of $S_1$. To see the admissibility of $g_1$, we consider two cases. First, assume that the image of $g_1$ is contained in $|S_1^-|$. Then, $g_1$ satisfies the modulus condition. Next, assume that the image of $g_1$ is not contained in $|S^-|$. Since by the minimality of $f$ we have $S^- = f^\ast T^-$, this implies that the image of $f \circ g_1 = g \circ f_1$ is not contained in $T^-$. In particular, the image of $g$ is not contained in $T^-$. Therefore, the admissibility of $g$ implies  
\[
T_1^+ |_{\ol{T}_1^N} - T_1^- |_{\ol{T}_1^N} \geq g^\ast T^+ |_{\ol{T}_1^N} -  g^\ast T^- |_{\ol{T}_1^N},
\]
where $\ol{T}_1^N$ is the normalization of $\ol{T}^1$. Noting that there exists a canonical morphism $\ol{S}_1^N \to \ol{T}_1^N$, we obtain 
\[
T_1^+ |_{\ol{S}_1^N} - T_1^- |_{\ol{S}_1^N} \geq g^\ast T^+ |_{\ol{S}_1^N} - g^\ast T^- |_{\ol{S}_1^N}.
\]
By the minimality of $f_1$ and $f$, we obtain 
\[T_1^\pm |_{\ol{S}_1^N} = S_1^\pm |_{\ol{S}_1^N}, \ \ T^\pm |_{\ol{S}_1^N} = S^\pm |_{\ol{S}_1^N}.\]
Therefore, we have 
\[
S_1^+ |_{\ol{S}_1^N} - S_1^- |_{\ol{S}_1^N} \geq g^\ast S^+ |_{\ol{S}_1^N} - g^\ast S^- |_{\ol{S}_1^N},
\]
which shows that $g_1$ is admissible.

Now, we prove that $S_1$ represents the fiber product $S \times_T T_1$. Give ourselves a commutative diagram in $\ulTSm^\fin$:
\[\xymatrix{
U \ar[r]^{a} \ar[d]_{b} & S \ar[d]^f \\
T_1 \ar[r]^g & T.
}\]
Then, the underlying diagram of schemes induces a unique morphism $c : \ol{U} \to \ol{S} \times_{\ol{T}} \ol{T}_1$ which makes the resulting diagram commute. It suffices to prove that $c$ defines an admissible morphism $U \to S_1$. If the image of $c$ is contained in $S_1^- = \ol{S} \times_{\ol{T}} T^-$, then we are done by Lemma \ref{mod-criterion1}. Assume that the image of $c$ is not contained in $S_1^- = \ol{S} \times_{\ol{T}} T^-$. Then, the image of $b$ is not contained in $T_1^-$. In this case, the admissibility of $b$ implies
\[
U^+ |_{\ol{U}^N} - U^- |_{\ol{U}^N} \geq T_1^+ |_{\ol{U}^N} - T_1^- |_{\ol{U}^N},
\]
where $\ol{U}^N$ is the normalization of $\ol{U}$. Noting that $S_1^\pm |_{\ol{U}^N} = T_1^\pm |_{\ol{U}^N}$, we obtain 
\[
U^+ |_{\ol{U}^N} - U^- |_{\ol{U}^N} \geq S_1^+ |_{\ol{U}^N} - S_1^- |_{\ol{U}^N},
\]
which shows that $c$ is admissible. 
\end{proof}

\section{Important subcategories}\label{s3}

In this section, we introduce several interesting categories and study their relation to $\ulTCor$.

\subsection{Proper modulus triples}\label{subsection-subcats}

\begin{definition}
Define $\TCor \subset \ulTCor$ to be the full subcategory consisting of those $T \in \ulTCor$ such that $\ol{T}$ is proper over $k$. Let $\tau : \TCor \to \ulTCor$ be the inclusion functor, which is monoidal by a trivial reason. 
\end{definition}


The aim of this subsection is to prove the following result.

\begin{thm}\label{adjoints-tau}
The functor $\tau$ has a pro-left adjoint $\tau^!$,  which is given by the formula 
\[
\tau^! (T) := \proobj{(T \to T_1) \in \Comp(T)} T_1,
\] 
where the definition of $\Comp(T)$ is given below.
\end{thm}



\begin{definition}\label{comp-1}
For any $T \in \ulTCor$, define $\Comp(T)$ to be the category whose objects are pairs $(T_1,j)$ consisting of $T_1 \in \TCor$ equipped with a dense open immersion $j:\ol{T} \hookrightarrow \ol{T}_1$ such that $T_1^+ = T_0^+ + C$ for some effective Cartier divisors $T_0^+, C$ on $\ol{T}_1$ satisfying $\ol{T}_1 \setminus |C| = j(\ol{T})$ and $j$ induces a minimal morphism $j:T \to T_1$ in $\ulTSm$ (see Definition 
\ref{def:dommin} for the minimality).
For $T_1,T_2 \in \Comp(T)$, the set of morphisms is given by
\[
\Comp(T)(T_1,T_2) := \defset{\alpha \in \TCor(T_1,T_2)}{23mm}{\ $j_{T_2}=\alpha \circ j_{T_1}$}.
\]
Note that $\Comp(T)$ is an essentially small category.
\end{definition}


\begin{prop}\label{prop-comp1}
For any $T \in \ulTCor$, the category $\Comp(T)$ is non-empty, cofiltered and ordered. More precisely, for any $(j_{T_{1}} : T \to T_1) , (j_{T_{2}} : T \to T_2) \in \Comp(T)$, we can always find $(j_{T_{1}} : T \to T_3) \in \Comp(T)$ which dominates these objects, such that the morphisms $T_3 \to T_1$ and $T_3 \to T_2$ come from $\ulTSm^\fin $ and the following condition holds: 

$\mathrm{(\star)}$ $\ol{T}_3 \setminus j_{T_3}(\ol{T})$ is the support of an effective Cartier divisor on $\ol{T}_3$.
\end{prop}

\begin{proof} 
 It is similar to  that of \cite[Lemma 1.8.2]{motmodI}, but a bit more involved.
We first prove that $\Comp(T)$ is non-empty and ordered. Take any dense open immersion $j_0 : \ol{T} \to \ol{T}_0$  with $\ol{T}_0$ proper over $k$.
 Let  $T_0^\pm$ be the scheme-theoretic closure of $T^\pm$. Consider the succession of blow-ups 
\[\pi : T_1 := \ol{T}''' \xrightarrow{\pi_C} \ol{T}''  \xrightarrow{\pi_-} \ol{T}' \xrightarrow{\pi_+} \ol{T}_0,\]
where $\pi_+$ is the blow-up at $T_0^+$, $\pi_-$ is the blow-up at $\pi_+^{-1}(T_0^-)$, and $\pi_C$ is the blow-up at $(\pi_+ \circ \pi_0)^{-1}(\ol{T}_0 \setminus j_0(\ol{T}))$.
Then, $\pi$ is an isomorphism over $\ol{T}$, therefore the map $j_0$ lifts to a dense open immersion $j_1 : \ol{T} \to \ol{T}_1$. 
By construction, $T_1^\pm := \pi^\ast T_0^\pm$ are effective Cartier divisors on $\ol{T}_1$. 
Moreover, $C:=\pi^{-1}(\ol{T}_0 \setminus j_0(\ol{T}))$ is an effective Cartier divisor contained in $T_1^+$.
Set $T_1 := (\ol{T}_1,T_1^+,T_1^-)$. Then, $j_1$ defines a minimal morphism $T \to T_1$ which induces the identity on the interiors. Thus, we obtained an object $(T \to T_1)$ in $\Comp(T)$.
The fact that all morphisms in $\Comp(T)$ are the identity on the interiors implies that $\Comp(T)$ is ordered.

We prove the rest of the assertions. Take any $(j_{T_{1}} : T \to T_1) , (j_{T_{2}} : T \to T_2) \in \Comp(T)$. Then, the total spaces $\ol{T}_1$ and $\ol{T}_2$ are birational and have $\ol{T}$ as a common open dense subset. Let $\Gamma$ be the graph of the birational map $\ol{T}_1 \dashrightarrow \ol{T}_2$. Note that $j_{T_1}$ and $j_{T_2}$ both lift to a common open dense immersion $j : \ol{T} \to \Gamma$. Set $Z:=(\Gamma \setminus j(\ol{T}))_\red$, and consider the blow up $T_3 := \pi : \Bl_{Z}(\Gamma) \to \Gamma$ of $\Gamma$ along $Z$. Then, $j$ lifts to an open dense immersion $j_{T_3} : \ol{T} \to \ol{T}_3$. Let $p:\ol{T}_3 \to \ol{T}_1$ be the natural projection, and let $E$ be the exceptional divisor, i.e., $E=\pi^{-1} (Z)$. Fix a positive integer $n$ and set
\[
T_3 := (\ol{T}_3 , p^\ast T_1^+ + nE, p^\ast T_1^-).
\]
Then we have $\ol{T}_3 \setminus j_{T_3} (\ol{T}) = p^{-1} (\ol{T}_1 \setminus j_{T_1}(\ol{T}))$, which is the support of an effective Cartier divisor.
Since $j_{T_3}^\ast (p^\ast T_1^+ + nE) = T^+$ and $j_{T_3}^\ast (p^\ast T_1^-) = T^-$ by the minimality of $j_{T_1}$, we have $(T \to T_3) \in \Comp(T)$. It is obvious that the projection $p$ gives a morphism $(T \to T_3) \to (T \to T_1)$ in $\Comp(T)$. It remains to prove that there exists a positive integer $n_0$ such that for any $n \geq n_0$ the projection $q : \ol{T} \to \ol{T}_2$ also defines a morphism $(T \to T_3) \to (T \to T_2)$ in $\Comp(T)$. The projection $q$ defines an admissible morphism $T_3 \to T_2$ if 
\[
(p^\ast T_1^+ + nE) - p^\ast T_1^- \geq q^\ast T_2^+ - q^\ast T_2^-,
\]
which is equivalent to that 
\[
nE \geq -p^\ast (T_1^+ - T_1^-) + q^\ast (T_2^+ - T_2^-).
\]
The right hand side becomes zero on the open subset $\ol{T}$ by the minimality of $j_{T_1}$ and $j_{T_1}$. Therefore, the support of the right hand side is contained in $|E|$. Therefore, applying \cite[Lemma 3.16]{cubeinv}, 
we can find the desired $n_0$. 
\end{proof}

Now we are ready to prove Theorem \ref{adjoints-tau}.

\begin{proof}[Proof of Theorem \ref{adjoints-tau}]
For any $T \in \ulTCor$, define a pro-object $\tau^! (T) \in \pro{}\TCor$ by 
\[
\tau^! (T) := \proobj{(T \to T_1) \in \Comp(T)} T_1.
\]
Then, we have an evident morphism $T \to \tau \tau^! (T)$ in $\pro{}\ulTCor$, where we regard $T$ as a constant pro-object. It suffices to prove that for any $S \in \TCor$, the induced map of abelian groups 
\[
\varinjlim_{(T \to T_1)\in \Comp_1(T)} \TCor (T_1,S) = \pro{}\TCor (\tau^! (T),S) \to \ulTCor (T,S).
\]
is an isomorphism. The injectivity is obvious since both sides are subgroups of $\Cor(T^\o,S^\o)$. To prove the surjectivity, take any integral cycle $V \in \ulTCor (T,S)$. If $V$ is contained in the fiber of $S^-$, then $V \in \TCor (T_1,S)$ for any $(T \to T_1)\in \Comp(T)$ and we are done. We assume that $V$ is not contained in the fiber of $S^-$.
Let $(T \to T_1) \in \Comp(T)$ be an object which satisfies the condition $(\ast)$ in Proposition \ref{prop-comp1}. Let $\ol{V} \subset \ol{T}_1 \times \ol{S}$ and let $\ol{V}^N$ be its closure. Since $V \in \ulTCor (T,S)$, the Cartier divisor 
\[
(T_1^+ |_{\ol{V}^N} - T_1^- |_{\ol{V}^N} ) - (S^+ |_{\ol{V}^N} - S^- |_{\ol{V}^N})
\]
is effective on the open subset $\ol{V}^N \times_{\ol{T}_1} \ol{T} \subset \ol{V}^N$. Therefore, by \cite[Lemma 3.12]{cubeinv} (or \cite[Lemma B.1]{KM}), we can find a positive integer $n$ such that 
\[
(T_1^+ |_{\ol{V}^N} - T_1^- |_{\ol{V}^N} ) - (S^+ |_{\ol{V}^N} - S^- |_{\ol{V}^N}) + nD \geq 0,
\]
where $D$ is an effective Cartier divisor on $\ol{T}_1$ such that $|D|=\ol{T}_1 \setminus \ol{T}$. Define $T_2 := (\ol{T}_1 , T_1^+ + nD , T_1^-)$. Then, the open immersion $\ol{T} \to \ol{T}_1$ defines a minimal morphism $T \to T_2$ and we have $(T \to T_2) \in \Comp(T)$. Moreover, $V \in \TCor(T_2,S)$. 
\end{proof}

\subsection{Disjoint modulus triples}

The main aim of this subsection is to study the properties of the following interesting category:

\begin{definition}\label{d3.5}
We write $\ulTphiCor$ for the full category of $\ulTCor$ consisting of interiorly smooth modulus triples $T=(\ol{T},T^+,T^-)$ such that $T^+$ and $T^-$ are disjoint. 
Moreover, we set $\TphiCor := \ulTphiCor \cap \TCor$.
\end{definition}

\begin{remark}\label{r3.1}
The most prominent example of an object in $\ulTphiCor$ would like to be the modulus triple $\Bl(T)$ associated to a modulus triple $T$ in Definition \ref{d2.2}, cf. Lemma \ref{l2.1}. Unfortunately, it is not necessarily an object of $\ulTphiCor$ since its interior might not be smooth even if $T \in \ulTCor$ and $\ol{T}$ is smooth. Here is an example:

Take $T = ( \ol{T} , T^+ , T^- ) $
with
\begin{itemize}
\item $\ol{T} = \A^2=\Spec k[x,y]$
\item $T^+ =$ Cartier divisor $(y=0)$
\item $T^- =$ Cartier divisor $(x^2=0)$.
\end{itemize}

Then the total space of $\Bl(T)$ is $\Bl_{(x^2,y)}(\A^2)$, which has a singular  point outside the strict transform of $T^+$ (compare \cite[pp. 179--180]{eis-har}).
Hence $\Bl(T)^\o$ is singular.

 This is corrected by Theorem \ref{pro-separation} below, which is the main result of this subsection.
\end{remark}

\begin{thm}\label{pro-separation}
Assume that $k$ admits resolution of singularities. 
Then, the full embedding $\ul{t} : \ulTphiCor \to \ulTCor$ admits a pro-left adjoint $\ul{s}$.
\end{thm}

Before starting the proof of Theorem \ref{pro-separation}, we prepare some definitions and lemmas. 

\begin{defn} Define two full subcategories $\ulTCor^*$ and $\ulTCor^{**}$ of $\ulTCor$ by
\begin{align*}
\ulTCor^*&=\{T\in \ulTCor\mid  F_T=T^+\times_{\ol{T}} T^- \text{ is an effective Cartier divisor}\}\\
\ulTCor^{**}&=\{T\in \ulTCor^*\mid  \ol{T} \setminus |T^+-F_{T}| \text{ is smooth over } k.\}
\end{align*}
\end{defn}

\begin{lemma}\label{l3.2} For any $T_0\in\ulTCor$, there is a morphism $p:T\to T_0$ in $\Sigma^\fin$ with $T\in \ulTCor^*$. (See Definition \ref{localization1} for $\Sigma^\fin$.) In particular, the full embedding $\ulTCor^*\inj \ulTCor$ is an equivalence of categories.
\end{lemma}

\begin{proof} Let $\ol{T}=\Bl_{F_{T_0}}(\ol{T_0})$ and let $p:\ol{T}\to \ol{T_0}$ be the projection. Then $p$ induces a morphism
\[T:=(\ol{T},p^*T_0^+,p^*T_0^-)\to T_0\]
and $F_T=p^*F_{T_0}$ is an effective Cartier divisor. This morphism is clearly in $\Sigma^\fin$. The last statement is obvious, since morphisms of $\Sigma^\fin$ are invertible in $\ulTCor$.
\end{proof}

\begin{lemma}\label{l3.3} Let $T\in \ulTCor^*$. Under resolution of singularities, there exists a morphism $\pi:T_\pi\to T$ in $\Sigma^\fin$ with $T_\pi\in \ulTCor^{**}$. In particular, the full embedding $\ulTCor^{**}\inj \ulTCor$ is an equivalence of categories.
\end{lemma}

\begin{proof} Simply resolve the singularities of $\ol{T}$, and pull back.
\end{proof}

Let $\ulTCor^{**,\fin}=\ulTCor^{**}\cap \ulTCor^\fin$ (a full subcategory of $\ulTCor^\fin$).
Let $T\in\ulTCor^{**,\fin}$. Define $T'=(\ol{T},T^+-F_T,T^--F_T)$: this is an object of $\ulTCor^\fin$ by hypothesis, and even of $\ulTphiCor^\fin$ by Lemma \ref{key-lem}. We have a canonical morphism in $\ulTCor^\fin$
\begin{equation}\label{eq3.4}
T\by{\iota} T'
\end{equation}
which is the identity of $\ol{T}$.

\begin{lemma}\label{claim-4.8.1}
Let $S \in \ulTphiCor^\fin$; take any irreducible correspondence $V \in \ulTCor^\fin (T,S)$. 
Let $V'$ (resp.  $\ol{V}$) be the closure of $V$ in ${T'}^\o \times S^\o$ (resp. in $\ol{T} \times \ol{S}$). Then
\begin{thlist}
\item  $\ol{V}$ satisfies the modulus condition for $T' \otimes S^\vee$.
\item  $V'$ is  finite over ${T'}^\o$. In particular, $V\iso V' \times_{{T'}^\o} T^\o$.
\end{thlist}
\end{lemma}

\begin{proof}
If $\ol{V}\subseteq \ol{T}\times S^-$, (i) holds by Lemma \ref{fully-contained2}. Otherwise, the modulus condition on $V$ implies that 
\[
T^+ |_{\ol{V}^N} + S^- |_{\ol{V}^N} \geq T^- |_{\ol{V}^N} + S^+ |_{\ol{V}^N} 
\]
which is equivalent to
\begin{equation}\label{eq4.3.2}
(T')^+ |_{\ol{V}^N} + S^- |_{\ol{V}^N} \geq (T')^- |_{\ol{V}^N} + S^+ |_{\ol{V}^N} .
\end{equation}

This proves (i).

Now we prove (ii). Note that the following inclusion holds:
\begin{equation}\label{eq4.3.3}
\ol{V} \cap (\ol{T} \times |S^+|) \subset \ol{V} \cap (|(T')^+| \times \ol{S}).
\end{equation}
Indeed, if $\ol{V}\subseteq \ol{T}\times S^-$, then the assumption that $S^+ \cap S^- = \emptyset$ implies that $\ol{V} \cap (\ol{T} \times |S^+|) = \emptyset$. If $\ol{V}\not\subset \ol{T}\times S^-$, then the assumption that $S^+ \cap S^- = \emptyset$ and the inequality \eqref{eq4.3.2} proven above imply that $(T')^+ |_{\ol{V}^N} \geq S^+ |_{\ol{V}^N}$, whence \eqref{eq4.3.3}.

By taking the complements of both sides of \eqref{eq4.3.3} in $\ol{V}$, we obtain
\[
\ol{V} \cap (\ol{T} \times S^\o ) \supset \ol{V} \cap ({T'}^\o \times \ol{S}),
\]
which implies 
\[
\ol{V} \cap ({T'}^\o \times \ol{S}) = \ol{V} \cap ({T'}^\o \times \ol{S}) \cap (\ol{T} \times S^\o ) = \ol{V} \cap ({T'}^\o \times S^\o) = V' .
\]

Since $\ol{V} \cap ({T'}^\o \times \ol{S})$ is  finite over ${T'}^\o$, we conclude that $V'$ is  finite over ${T'}^\o$.
In particular, $V' \times_{{T'}^\o} T^\o$ is  finite over $T^\o$, therefore the open immersion $V \inj V' \times_{{T'}^\o} T^\o$ is an isomorphism since $V$ is finite over $T^\o$ and dense in $V'$.  
\end{proof}

Consider the commutative diagram of categories and functors
\begin{equation}\label{eq3.2}
\begin{CD}
\ulTphiCor^\fin @>\ul{t}^\fin>> \ulTCor^\fin\\
@Vb^\emptyset VV @VbVV\\
\ulTphiCor @>\ul{t}>> \ulTCor
\end{CD}
\end{equation}
where $b$ is the inclusion functor of Proposition \ref{localization1} (b) and $b^\emptyset$ is its restriction fo $\ulTphiCor^\fin$. 

\begin{lemma} \label{l3.4} 
The left adjoint $\ul{s}^\fin$ of $\ul{t}^\fin$ is defined at all objects of $\ulTCor^{**,\fin} \subset \ulTCor^{\fin}$.
\end{lemma}

\begin{proof} Let $T\in \ulTCor^{**,\fin}$. 
We claim that the morphism $\iota$ of \eqref{eq3.4} induces an isomorphism 
\[
\ulTphiCor^\fin (T',S) \xrightarrow{\sim} \ulTphiCor^\fin (T,S) = \ulTCor^\fin (T,S)
\]
for any $S \in \ulTCor^{\fin}$, which shows that the assignment $T \mapsto T'$ gives the left adjoint $\ul{s}^{\fin}$.

Let $S\in \ulTphiCor^\fin$ and let $\alpha\in \ulTCor^\fin(T,S)$. 
Lemma \ref{claim-4.8.1} (ii) shows that  $\alpha$ extends to a finite correspondence $\alpha'\in \Cor({T'}^\o,S^\o)$ via $\iota$, and Lemma \ref{claim-4.8.1} (i) shows that $\alpha'\in \ulTCor^\fin(T',S)$.
Since $T^\o\to {T'}^\o$ is a dense open immersion, $\Cor({T'}^\o,S^\o)\by{\iota^*} \Cor(T^\o,S^\o)$ is injective, which shows the uniqueness of the extension $\alpha'$ and concludes the proof.
\end{proof}

\begin{proof}[Proof of Theorem \ref{pro-separation}] By Proposition \ref{localization1} (b) and \cite[Prop. A.6.2]{motmod}, the functor $b$ of \eqref{eq3.2} has a pro-left adjoint $b^!$ given by the formula
\[b^! T = \proobj{\Sigma^\fin\downarrow T} \tilde T\]
for $T\in \ulTCor$, where $\Sigma^\fin \downarrow T$ is the category of arrows $\tilde T\by{s} T$ with $s\in \Sigma^\fin$, morphisms being taken in $\ulTCor^\fin$.

Let $(\Sigma^\fin \downarrow T)^{**}$ denote the full subcategory of $\Sigma^\fin \downarrow T$ consisting of those $\tilde T\by{s} T$ such that $\tilde T\in \ulTCor^{**}$; lemmas \ref{l3.2} and \ref{l3.3} show that it is cofinal in $\Sigma^\fin \downarrow T$. Thus, we also have
\begin{equation}\label{eq3.3}
b^! T \simeq \proobj{(\Sigma^\fin\downarrow T)^{**}} \tilde T.
\end{equation}

We can then use Lemma \ref{l3.4} to apply $\ul{s}^\fin$ termwise to the right hand side of \eqref{eq3.3}, yielding a functor
\[\ul{s}^\fin b^!:\ulTCor\to \pro{}\ulTphiCor^\fin\]
which is pro-left adjoint to the composite $b\ul{t}^\fin\simeq \ul{t} b^\emptyset$. 
The unit morphism of this adjunction may be written as
\[\eta_T:T\to b\ul{t}^\fin \ul{s}^\fin b^!T \simeq \ul{t} b^\emptyset \ul{s}^\fin b^! T.\]

Let us show that this morphism makes $b^\emptyset \ul{s}^\fin b^!$ a pro-left adjoint of $\ul{t}$. Since $b^\emptyset$ is the identity on objects, we have a commutative diagram for $T\in \ulTCor$ and $S\in \ulTphiCor^\fin$:
\[\begin{CD}
\pro{}\ulTphiCor^\fin(\ul{s}^\fin b^!T,S)@>\sim >> \ulTCor(T,\ul{t} b^\emptyset S)\\
@Vb^\emptyset VV @V{||}VV\\ 
\pro{}\ulTphiCor(b^\emptyset\ul{s}^\fin b^!T,b^\emptyset S)@>\eta_T^*\circ \ul{t} >> \ulTCor(T,\ul{t}  S)
\end{CD}\]
in which the top row is an isomorphism by the above-explained adjunction; this implies that the bottom row is surjective, and it remains to show that it is injective. Since $\ul{t}$ is fully faithful, we are left to show that $\eta_T^*$ is injective, and for this it suffices to show that $\ulomega(\eta_T)^*$ is injective, where $\ulomega:\ulTCor\to \Cor$ is the faithful functor of \eqref{eq1.1}. This is true because $\ulomega(\eta_T)$ is termwise a dense open immersion (this argument was already used in the proof of Lemma \ref{l3.4}).
\end{proof}

\begin{rk} One can replace the geometric argument at the end of this proof by a formal argument dual to the proof of \cite[Prop. 4.3.6 (b)]{2017.1}, noting that, like $b$,  $b^\emptyset$ is a localization and has a pro-left adjoint. We refrain from doing this because this formal argument would make the proof less transparent. Nevertheless, the concrete description of $\ulomega(\eta_T)$ is not essential here.
\end{rk}

\begin{rk} We don't know if the left adjoint of $\ul{t}^\fin$ is defined at all objects $T\in \ulTCor^\fin$; if it were, this would potentially allow us to remove the resolution of singularities hypothesis. In the case $T=(\ol{T},T^+,T^-)$ with $T^+=T^-$, this leads to the following intriguing question: is the functor
\[\Cor\ni S^\o\mapsto \Cor(\ol{T},S^\o)\]
corepresentable? (See Lemma \ref{l3.5} below.)
\end{rk}

\begin{lemma}\label{l3.5}
Let $T= (\ol{T},T^+,T^-) \in \ulTCor$, and assume $T^+=T^-$. Let $S \in \ulTphiCor$. Then, we have an identification
\[
\ulTCor^\fin (T,S) \cong \Z \defset{V \in \Cor(T^\o,S^\o)}{45mm}{$V$ is irreducible, and the closure of $V$ in $\ol{T} \times S^\o$ is finite over $\ol{T}$}.
\]
\end{lemma}

\begin{proof}
Since both sides are subgroups of $\Cor(T^\o,S^\o)$, it suffices to prove that for any integral cycle $V \in \Cor(T^\o,S^\o)$, the closure $\ol{V}$ of $V$ in $\ol{T} \times \ol{S}$ satisfies the modulus condition for $T \otimes S^\vee$ if and only if $\ol{V} \subset \ol{T} \times S^\o$. This can be checked as follows: 
\begin{align*}
&\text{$\ol{V}$ satisfies the modulus condition for $T \otimes S^\vee$} \\
&\Leftrightarrow^1 \text{$\ol{V} \subset \ol{T} \times |S^-|$\  or\  $T^+|_{\ol{V}^N} + S^-|_{\ol{V}^N} \geq S^+|_{\ol{V}^N} + T^-|_{\ol{V}^N}$} \\
&\Leftrightarrow^2 \text{$\ol{V} \subset \ol{T} \times |S^-|$\  or\  $S^-|_{\ol{V}^N} \geq S^+|_{\ol{V}^N}$} \\
&\Leftrightarrow^3 \text{$\ol{V} \subset \ol{T} \times |S^-|$\  or\  $S^+|_{\ol{V}^N} = \emptyset$} \\
&\Leftrightarrow^4 \text{$\ol{V} \subset \ol{T} \times |S^-|$\  or\  $S^+|_{\ol{V}} = \emptyset$} \\
&\Leftrightarrow^5 \text{$\ol{V} \subset \ol{T} \times |S^-|$\  or\  $\ol{V} \subset \ol{T} \times S^\o$} \\
&\Leftrightarrow^6 \text{$\ol{V} \subset \ol{T} \times S^\o$},
\end{align*}
where $\Leftrightarrow^1$ follows from Lemma \ref{general-mod} and Lemma \ref{fully-contained2}, $\Leftrightarrow^2$ from $T^+=T^-$, $\Leftrightarrow^3$ from $S^+ \cap S^- = \emptyset$, $\Leftrightarrow^4$ from the surjectivity of $\ol{V}^N \to \ol{V}$, $\Leftrightarrow^5$ is obvious, and $\Leftrightarrow^6$ follows again from $S^+ \cap S^- = \emptyset$ (which is equivalent to $|S^-| \subset S^\o$). 
\end{proof}

\subsection{Relation to $\protect\Cor$}\label{section-CorTCor}

\begin{definition}\label{d3.2}
Define a functor $\ul{\omega} : \ulTCor \to \Cor$ by 
\[
\ul{\omega}(T) := T^\o.
\] 
Moreover, define a functor $\lambda : \Cor \to \ulTCor$ by 
\[
\lambda(X) := (X,\emptyset,\emptyset).
\]
It is easy to check the functorialities.
\end{definition}

\begin{prop}\label{p3.1}
The functor $\lambda$ is left adjoint to $\ul{\omega}$.
\end{prop}

\begin{proof}
Let $X \in \Cor$ and $T \in \ulTCor$. It suffices to check that 
\[
\ulTCor ((X,\emptyset,\emptyset),T) = \Cor(X,T^\o).
\]
The left hand side is obviously contained in the other.
Let $V \in \Cor(X,T^\o)$ be any elementary finite correspondence.
Since $V$ is finite (hence proper) over $X=(X,\emptyset,\emptyset)^\o$, it suffices to check the modulus condition.
If $V \subset X \times |T^-|$, then we are done by Proposition \ref{mod-criterion1}.
Assume that $V \not\subset X \times |T^-|$.
Note that $V$ is closed in $X \times \ol{T}$ since $V$ is proper over $X$. Let $V^N$ be the normalization of $V$.
Then we have to check that the following inequality holds:
\[
T^- |_{V^N} \geq T^+ |_{V^N}.
\]
However, since $V \subset X \times T^\o$, the right hand side is zero. 
\end{proof}

\subsection{Relation to $\protect\ulMCor$}\label{section-MCorTCor}

One of the aim of this paper is to enlarge the category of motives with modulus of \cite{motmod}. The fundamental building blocks of this theory are the two types of categories of modulus pairs, denoted $\ulMCor$ and $\MCor$, which were introduced in \cite{motmodI}. Our task in this subsection is to compare these categories with the categories of modulus triples.




Consider the following commutative diagram of monoidal and fully faithful functors:
\begin{equation}\label{eq3.5}
\begin{gathered}
\xymatrix{
\ulMCor \ar[rr]^{\ul{\phi}} \ar[dr]_{\ul{\psi}} && \ulTphiCor \ar[dl]^{\ul{t}} \\
& \ulTCor, &
}
\end{gathered}
\end{equation}
where $\ul{\phi}(M) = \ul{\psi}(M) = (\ol{M},M^\infty,\emptyset )$, and $\ul{t}(T) = T$.


\begin{thm}\label{adjoint-phi}
$\ul{\phi}$ admits a left adjoint $\ul{p}$ and a right adjoint $\ul{q}$.
\end{thm}

\begin{proof}
We first prove that $\ul{\phi}$ admits a left adjoint. For any $T \in \ulTphiCor$, 
 let $T \cong \oplus_{i\in I} T_i$ be the canonical decomposition of $T$ as in Example \ref{ex2.2} and set
\[
\ul{p}(T) := \bigoplus_{i \in I, T_i^- = \emptyset} (\ol{T}_i,T_i^+).
\]

Then, the natural projection  
\[T^\o = \bigoplus_{i\in I} T_i^\o \to \bigoplus_{i\in I, T_i^-=\emptyset} T_i^\o = (\ul{p}(T))^\o \]
induces a morphism $T \to \ul{p}(T)$ in $\ulTphiCor$. It induces for any $M \in \ulMCor$ a map of abelian groups
\[\ulMCor (\ul{p}(T),M) = \ulMCor (\ul{\phi}\ul{p}(T),\ul{\phi}M) \to \ulTphiCor (T,\ul{\phi}M).\]
It suffices to prove that this is an isomorphism for any $T \in \ulTphiCor$ and $M \in \ulMCor$. This is obvious if $T^-=\emptyset$. Assume that $T^-\neq \emptyset$. Then, we have $\ulMCor (p(T),M)=\ulMCor (0,M)=0$. Therefore, we are reduced to showing that $\ulTphiCor (T,\ul{\phi}M) = 0$, which follows from the fact that $(\ul{\phi}M)^-=\emptyset$ and Proposition \ref{-goes-}. 

We prove the existence of a right adjoint of $\ul{\phi}$. For any $T \in \ulTphiCor$, define an object $\ul{q}(T)$ in $\ulMCor$ by setting $\ul{q}(T)=(\ol{T},T^+)$. Then, the identity on $\ol{T}$ induces a morphism $\ul{\phi}\ul{q}(T) = (\ol{T},T^+,\emptyset ) \to T$ in $\ulTphiCor$. It induces for any $M \in \ulMCor$ a map of abelian groups
\[
\ulMCor (M,\ul{q}(T)) = \ulMCor (\ul{\phi}M,\ul{\phi}\ul{q}(T)) \to  \ulTphiCor (\ul{\phi}M,T).
\]
It suffices to prove that this is an isomorphism. Let $V \in \Cor (M^\o , T^\o)$ be an elementary finite correspondence. Let $\ol{V}$ be the closure of $V$ in $\ol{M} \times \ol{T}$, and let $\ol{V}^N$ be its normalization. Then, $V$ belongs to $\ulMCor(\ul{\phi}M,T)$ if and only if $\ol{V}$ is proper over $\ol{M}$ and the following inequality holds:
\begin{equation}\label{eq4.1.1}
M^\infty |_{\ol{V}^N} + T^- |_{\ol{V}^N} \geq T^+ |_{\ol{V}^N}.
\end{equation}
On the other hand, $V$ belongs to $\ulMCor (M,\ul{q}(T))$ if and only if $\ol{V}$ is proper over $\ol{M}$ and the following inequality holds:
\begin{equation}\label{eq4.1.2}
M^\infty |_{\ol{V}^N} \geq T^+ |_{\ol{V}^N}.
\end{equation}
By the assumption that $T^+ \cap T^- = \emptyset$, the conditions \eqref{eq4.1.1} and \eqref{eq4.1.2} are equivalent. 
\end{proof}

\begin{thm}\label{adjoint-psi} 
Assume that $k$ admits the resolution of singularities. 
Then, $\ul{\psi}$ admits a pro-left adjoint.
\end{thm}

\begin{proof}
This follows from Theorems \ref{pro-separation} and \ref{adjoint-phi}.
\end{proof}

\subsection{The coadmissible version: $\protect\ulWCor, \protect\WCor$}\label{s3.5}


\begin{definition}
Write $\ulWCor$ for the category whose objects are modulus pairs $M=(\ol{M},M^\infty )$ with $\ol{M}$ smooth over $k$ such that for any $M_1,M_2 \in \ulWCor$, we have \[\ulWCor (M_1,M_2) := \ulTCor ((\ol{M}_1,\emptyset ,M_1^\infty ),(\ol{M}_2,\emptyset ,M_1^\infty )).\]
The category $\ulWCor$ is called the category of \emph{coadmissible correspondences}.
Clearly, the category $\ulWCor$ is naturally identified with the full subcategory of $\ulTCor$ consisting of interiorly smooth modulus triples of the form $(\ol{T},\emptyset,T^-)$.
\end{definition}

\subsection{Correspondences in bad position}

\begin{definition}
Define $\ulTbCor$ to be the sub-quiver of $\ulTCor$ with same objects such that for any $S,T \in \ulTCor$ we have
\[
\ulTbCor(S,T) := \defset{\alpha \in \ulTCor(S,T)}{50mm}{For any component $V$ of $\alpha$, the image of $V$ in $T^\o$ is contained in $T^\o \cap T^-$.}.
\]
For the well-definedness of the composition, see the following proposition.
\end{definition}

\begin{prop}\label{tb-ideal}\
\begin{enumerate}
\item Let $S \xrightarrow{\alpha} T \xrightarrow{\beta} U$ be a diagram in $\ulTCor$. Assume that one of $\alpha$ or $\beta$ is a morphism in $\ulTbCor$, then so is the composite $\beta \circ \alpha$.
\item Let $S_1\by{\alpha_1} T_1$, $S_2\by{\alpha_2} T_2$ be two morphisms in $\ulTCor$. If either $\alpha_1$ or $\alpha_2$ belongs to $\ulTbCor$, so does $\alpha_1\otimes \alpha_2$ (see Proposition \ref{p2.1} for $\otimes$).
\end{enumerate}
\end{prop}

\begin{proof}
(1) This follows from the contraposition of Claim \ref{claim1.22} (a),(b).  (2) By (1) and symmetry, we are left to show that $\alpha\otimes 1\in \ulTbCor$ if $\alpha\in \ulTbCor$, which is trivial from the definition of $\otimes$.
\end{proof}

\begin{remark}
This proposition means that $\ulTbCor$ is a ``two sided  $\otimes$-ideal'' in $\ulTCor$.
\end{remark}

\begin{remark}\label{rem-gp}
Let $\alpha \in \ulTCor(S,T)$ be a modulus correspondence.
If any component of $\alpha$ does not belong to $\ulTbCor$, then we say that $\alpha$ is \emph{in good position}.
However, this condition is not very nice, because the composition in $\ulTCor$ does not preserve it. We will introduce a stronger notion of ``very good position'' in \S \ref{section-gp}.
\end{remark}

\subsection{Reduced correspondences}

\begin{definition}\label{d3.1}
Define an additive category $\ulTredCor$ as the category whose objects are $\mathrm{Ob}(\ulTCor)$ and for any $S,T \in \ulTredCor$ the hom groups are given by 
\[
\ulTredCor (S,T) := \frac{\ulTCor(S,T)}{\ulTbCor(S,T)}.
\]

By Proposition \ref{tb-ideal}, the composition is well-defined and the tensor structure of $\ulTCor$ induces a tensor structure on $\ulTredCor$ via the canonical functor 
\begin{equation}\label{eq3.6}
\ul{\rho} : \ulTCor \to \ulTredCor.
\end{equation}
 We call $\ulTredCor$ the category of reduced correspondences.
\end{definition}

Note that the Hom group $\ulTredCor (S,T)$ remains a free abelian group for any $S,T$.

\subsection{Correspondences in very good position}\label{section-gp}

\begin{definition}\label{d3.4}
Let $S,T \in \ulTCor$ and let $\alpha \in \ulTCor(S,T)$. We say that $\alpha$ is \emph{in very good position} if 
\[V \cap (|S^-| \times T^\o) \supset V \cap (S^\o \times |T^-|).\]
We write $\ulTvgCor (S,T)$ be the free abelian group consisting of modulus correspondences in very good position.
\end{definition}

\begin{remark}\label{rem-3.27}
(1) By Lemma \ref{-goes-}, we always have the following inclusion:
\[V \cap (|S^-| \times T^\o) \subset V \cap (S^\o \times |T^-|).\]
Therefore, $V$ is in very good position if and only if we have equality.

(2) If $\alpha$ is in very good position, then it is in good position in the sense of Remark \ref{rem-gp}. Indeed, it is clear that $V \not\subset |S^-| \times T^\o$ for any component $V$ of $\alpha$.
\end{remark}

\begin{prop}\label{p3.2}
The composition in $\ulTCor$ preserves modulus correspondences in very good position.
\end{prop}

\begin{proof}
Let $T_1,T_2,T_3 \in \ulTCor$, and let $\alpha \in \ulTCor(T_1,T_2)$ and $\beta \in \ulTCor(T_2,T_3)$ be modulus correspondences in very good position. We want to prove that $\beta \circ \alpha$ is in very good position. Clearly we may assume that $\alpha$ and $\beta$ are integral cycles. Let $\gamma$ be any component of $\beta \circ \alpha$. By Remark \ref{rem-3.27}, it suffices to prove that $\gamma \cap (|T_1^-| \times T^\o_3) \supset \gamma \cap (T^\o_1 \times |T_3^-|)$.
As in Step 1 of the proof of Proposition \ref{p-composition}, we can find a component $\gamma'$ of $\alpha \times T_3^\o \cap T_1^\o \times \beta$ such that $\gamma'$ maps surjectively onto $\gamma$. Thus it suffices to prove that $\gamma' \cap (|T_1^-| \times T^\o_2 \times T^\o_3) \supset \gamma' \cap (T^\o_1 \times T^\o_2 \times |T_3^-|)$. The assumption that $\alpha$ and $\beta$ are in very good position implies 
\begin{align*}
(\alpha \times T^\o_3) \cap (|T_1^-| \times T^\o_2 \times T^\o_3) &\supset (\alpha \times T^\o_3) \cap (T^\o_1 \times |T_2^-| \times T^\o_3),\\
(T^\o_1 \times \beta) \cap (T^\o_1 \times |T_2^-| \times T^\o_3) &\supset (T^\o_1 \times \beta) \cap (T^\o_1 \times T^\o_2 \times |T_3^-|).
\end{align*}
These two inclusions immediately implies the desired one. 
\end{proof}

\begin{definition}
For any $S,T \in \ulTCor$, let $I(S,T)$ be the free abelian group generated by elementary correspondences $V \in \ulTCor(S,T)$ such that $V \notin \ulTvgCor(S,T)$. By definition, we have an evident equality 
\[
\ulTCor(S,T) = \ulTvgCor(S,T) \oplus I(S,T).
\] 
\end{definition}

\begin{prop}\label{p3.4}
Let $T_1,T_2,T_3 \in \ulTCor$, and let $\alpha \in \ulTCor(T_1,T_2)$ and $\beta \in \ulTCor(T_2,T_3)$. Assume that $\alpha \in I(T_1,T_2)$. Then we have $\beta \circ \alpha \in I(T_1,T_3)$.
In particular, we obtain a subcategory $I$ of $\ulTCor$.
\end{prop}

\begin{proof}
Clearly we may assume that $\alpha$ and $\beta$ are integral cycles. 
Assume that $\beta \circ \alpha \notin I(T_1,T_3)$. It suffices to prove that $\alpha \notin I(T_1,T_2)$, i.e., $\alpha$ is in very good position. By assumption, there exists a component $\gamma$ of $\beta \circ \alpha$ which is in very good position. As in Claim \ref{claim1.22} (c), we can find a component $\gamma'$ of $\alpha \times T_3^\o \cap T_1^\o \times \beta$ which maps surjectively onto both $\alpha$ and $\gamma$. The condition that $\gamma$ is in very good position implies
\begin{equation}\label{eq3.10}
\gamma' \cap (|T_1^-| \times T^\o_2 \times T^\o_3) \supset \gamma' \cap (T^\o_1 \times T^\o_2 \times |T_3^-|).
\end{equation}
On the other hand, Lemma \ref{-goes-} implies 
\[
(T^\o_1 \times \beta) \cap (T^\o_1 \times |T_2^-| \times T^\o_3) \subset (T^\o_1 \times \beta) \cap (T^\o_1 \times T^\o_2 \times |T_3^-|),
\]
hence, noting that $\gamma' \subset T^\o_1 \times \beta$, we have
\begin{equation}\label{eq3.11}
\gamma' \cap (T^\o_1 \times |T_2^-| \times T^\o_3) \subset \gamma' \cap (T^\o_1 \times T^\o_2 \times |T_3^-|).
\end{equation}
By \eqref{eq3.10} and \eqref{eq3.11}, we have
\[
\gamma' \cap (|T_1^-| \times T^\o_2 \times T^\o_3) \supset \gamma' \cap (T^\o_1 \times |T_2^-| \times T^\o_3).
\]
Since $\gamma' \to \alpha$ is surjective, we have 
\[
\alpha \cap (|T_1^-| \times T^\o_2) \supset \alpha \cap (T^\o_1 \times |T_2^-|).
\]
This shows that $\alpha$ is in very good position. 
\end{proof}

\begin{definition}
Define $\ulTvgCor$ as the subcategory of $\ulTCor$ such that the objects are interiorly smooth triples and the morphisms are modulus correspondences in very good position.
\end{definition}

\subsection{Correspondences in excellent position}

\begin{defn}\label{d2.3} Let $S,T \in \ulTCor$ and let $\alpha \in \ulTCor(S,T)$. We say that $\alpha$ is \emph{in excellent position} if, for any irreducible component $V$ of $\alpha$, we have the inclusion
\begin{equation}\label{eq3.16}
 \ol{V}\cap  (|S^-|\times \ol{T}) \supseteq \ol{V}\cap(\ol{S}\times |T^-|),
\end{equation}
or equivalently, 
\begin{equation}\label{eq:vvg}
|S^-|\times \ol{T} \supseteq \ol{V}\cap(\ol{S}\times |T^-|),
 \end{equation}
\end{defn}

\begin{lemma}\label{l3.6} If $V$ is in excellent position, it is in very good position (Definiiton \ref{d3.4}).
\end{lemma}

\begin{proof} Intersecting \eqref{eq3.16} with $V$, we get the inclusions
\[ V\cap  (|S^-|\times \ol{T}) \supseteq V\cap(\ol{S}\times |T^-|)\supseteq V\cap(S^\o\times |T^-|),
\]
but since $V\subseteq S^\o\times T^\o$, the left hand side equals $V\cap  (|S^-|\times T^\o)$.
\end{proof}

\begin{prop} The condition ``in excellent position'' is stable under composition, hence (by Lemma \ref{l3.6}) a subcategory $\ulTvvgCor\subset \ulTvgCor$.
\end{prop}

\begin{proof} 
 It is essentially the same as that of Proposition \ref{p3.2}, but a little more involved:

Let $\alpha \in \ulTCor (S,T)$ and $\beta \in \ulTCor (T,U)$ be finite correspondences whose irreducible components satisfy \eqref{eq:vvg}.
We may assume that $\alpha$ and $\beta$ are integral cycles. 
Consider the intersection $(\alpha \times U^\o) \cap (S^\o \times \beta) \subset S^\o \times T^\o \times U^\o$, and let $\gamma$ be any irreducible component of it.
Then, by the construction of the composition of finite correspondences, any irreducible component of $\beta \circ \alpha$ is of the form $\pi (\gamma)$ for some $\gamma$, where $\pi : S^\o \times T^\o \times U^\o \to S^\o \times U^\o$ denotes the projection. 
Therefore, it suffices to show that $\pi (\gamma)$ satisfies \eqref{eq:vvg}. 

Let $\ol{\alpha} \subset \ol{S} \times \ol{T}$,  $\ol{\beta} \subset \ol{T} \times \ol{U}$, $\ol{\gamma} \subset \ol{S} \times \ol{T} \times \ol{U}$ and $\ol{\pi (\gamma)} \subset \ol{S} \times \ol{U}$ be the closures.
Then the projection $\ol{\pi} : \ol{S} \times \ol{T} \times \ol{U} \to \ol{S} \times \ol{U}$ induces a proper dominant (hence surjective) morphism $\ol{\gamma} \to \ol{\pi (\gamma)}$, since $\ol{\gamma} \subset \ol{\alpha} \times \ol{U}$ is a closed immersion and $\ol{\alpha}$ is proper over $\ol{S}$ (by the left properness of $\alpha$). 

Since $\alpha$ and $\beta$ satisfy  \eqref{eq:vvg} by assumption, we have 
\begin{align*}
|S^-| \times \ol{T} &\supseteq \ol{\alpha} \cap (\ol{S} \times |T^-|), \\
|T^-| \times \ol{U} &\supseteq \ol{\beta} \cap (\ol{T} \times |U^-|).
\end{align*}
Pulling these back onto $\ol{S} \times \ol{T} \times \ol{U}$, we obtain 
\begin{align*}
|S^-| \times \ol{T} \times \ol{U} &\supseteq (\ol{\alpha} \times \ol{U}) \cap (\ol{S} \times |T^-| \times \ol{U}), \\
\ol{S} \times |T^-| \times \ol{U} &\supseteq (\ol{S} \times \ol{\beta}) \cap (\ol{S} \times \ol{T} \times |U^-|),
\end{align*}
and hence
\begin{align*}
|S^-| \times \ol{T} \times \ol{U} &\supseteq (\ol{\alpha} \times \ol{U}) \cap (\ol{S} \times |T^-| \times \ol{U}) \\
&\supseteq (\ol{\alpha} \times \ol{U}) \cap (\ol{S} \times \ol{\beta}) \cap (\ol{S} \times \ol{T} \times |U^-|) \\
&\supseteq \ol{\gamma} \cap (\ol{S} \times \ol{T} \times |U^-|).
\end{align*}
Since $\ol{\gamma} \to \ol{\pi (\gamma)}$ is surjective, the above inclusion implies 
\[
|S^-| \times \ol{U} \supseteq \ol{\pi (\gamma)} \cap (\ol{S} \times |U^-|),
\]
which shows that $\pi (\gamma)$ satisfies  \eqref{eq:vvg}, as desired.
\end{proof}

Note that the fully faithful functor $\ul{\psi} : \ulMCor \to \ulTCor$ of \eqref{eq3.5} restricts to a fully faithful functor ${}^{\mathrm{vvg}}\ul{\psi} : \ulMCor \to \ulTvvgCor$.

For any modulus triple $T$, we set
\[
T^{\o\o} := \ol{T} \setminus (|T^+| \cup |T^-|).
\]

\begin{prop} \label{p3.3}
There exists a functor $\ul{g} : \ulTvvgCor \to \ulMCor$ which assigns to a modulus triple $T$ a modulus pair 
\[
\ul{g}(T) := (\ol{T} \setminus |T^-|, T^+ |_{\ol{T} \setminus |T^-|}).
\]
\end{prop}

\begin{proof} Let $\alpha \in \ulTvvgCor(S,T)$ be an elementary modulus correspondence in excellent position. 
Define $\ul{g}(\alpha) := \alpha \cap (S^{\o\o} \times T^{\o\o})$. Since $\alpha$ is in very good position, we have $\ul{g}(\alpha) = \alpha \times_{S^\o} S^{\o\o}$, therefore $\alpha_0$ is finite and surjective over a component of $S^{\o\o}$. We will show that $\ul{g}(\alpha ) \in \ulMCor(\ul{g}(S),\ul{g}(T))$.

First we check that $\ul{g}(\alpha)$ satisfies the left properness condition. Let $\ol{\alpha} \subset \ol{S} \times \ol{T}$ and $\ol{\ul{g}(\alpha)} \subset (\ol{S} \setminus |S^-|) \times (\ol{T} \setminus |T^-|)$ be the closures. 
We have to prove that $\ol{\ul{g}(\alpha)}$ is proper over $\ol{S} \setminus |S^-|$.
Since the generic points of $\ol{\alpha}$ and $\ol{\ul{g}(\alpha)}$ are the same, we have $\ol{\ul{g}(\alpha)} = \ol{\alpha} \cap (\ol{S} \setminus |S^-|) \times (\ol{T} \setminus |T^-|)$.
Moreover, since $\ol{\alpha}$ is proper over $\ol{S}$, we know that $\ol{\alpha} \cap (\ol{S} \setminus |S^-|) \times \ol{T}$ is proper over $\ol{S} \setminus |S^-|$. Therefore, it suffices to show that $\ol{\alpha} \cap (\ol{S} \setminus |S^-|) \times \ol{T}$ is already contained in $(\ol{S} \setminus |S^-|) \times (\ol{T} \setminus |T^-|)$. But this is an immediate consequence of \eqref{eq:vvg}.

Next we show that $\ul{g}(\alpha)$ satisfies the modulus condition. 
Let $\ol{\alpha} \subset \ol{S} \times \ol{T}$ be the closure of $\alpha$ and $\ol{\alpha}^N$ its normalization. Then the modulus condition on $\alpha$ implies (noting that $\alpha \notin \ulTbCor(S,T)$) that 
\[
S^+ |_{\ol{\alpha}^N} + T^- |_{\ol{\alpha}^N} \geq S^- |_{\ol{\alpha}^N} + T^+ |_{\ol{\alpha}^N}.
\]
Let $\ol{\ul{g}(\alpha)} \subset (\ol{S} \setminus |S^-|) \times (\ol{T} \setminus |T^-|)$ be the closure of $\ul{g}(\alpha)$ and $\ol{\ul{g}(\alpha)}^N$ its normalization. Then the above inequality implies
\[
S^+ |_{\ol{\ul{g}(\alpha)}^N}  \geq T^+ |_{\ol{\ul{g}(\alpha)}^N},
\]
which shows that $\ul{g}(\alpha ) \in \ulMCor(\ul{g}(S),\ul{g}(T))$.

We define $\ul{g}(\alpha) \in \ulMCor(\ul{g}(S),\ul{g}(T))$ for any $\alpha \in \ulTvgCor (S,T)$ by linearly extending the above construction.

Finally we check that the association $\alpha \mapsto \ul{g}(\alpha)$ is compatible with the composition of $\ulTvgCor$ and $\ulMCor$. Let $\beta \in \ulTvgCor(T,U)$ for some modulus triples $U$. We want to prove that $\ul{g}(\beta \circ \alpha) = \ul{g}(\beta) \circ \ul{g}(\alpha)$. Recall that the composition of finite correspondences is defined by using the intersection product and the pushforward of algebraic cycles. Then the compatibility of them with the restriction to open subsets implies the desired equality. 
\end{proof}


\begin{prop}\label{p3.5}
The functor $\ul{g}$ is right adjoint to ${}^{\mathrm{vvg}}\ul{\psi}$.
\end{prop}

\begin{proof}
First note that for any $T \in \ulTvvgCor$ the natural inclusion $j : T^{\o\o} \to T^\o$ defines a morphism ${}^{\mathrm{vvg}}\ul{\psi}\ul{g}(T) \to T$ in $\ulTvvgCor$. Indeed, noting that $j$ extends to a morphism $\ol{T} \setminus |T^-| \to \ol{T}$, it is easy to check that the graph $\Gamma_j$ of $j$ belongs to $\ulTCor({}^{\mathrm{vvg}}\ul{\psi}\ul{g}(T), T)$. Moreover, $\Gamma_j$ does not intersect with $T^{\o\o} \times |T^-|$. Therefore, we have $\Gamma_j \in \ulTvvgCor({}^{\mathrm{vvg}}\ul{\psi}\ul{g}(T), T)$.
For any $M \in \ulMCor$ and $T \in \ulTvvgCor$, consider the map of abelian groups
\[
\ulMCor(M,\ul{g}(T)) \simeq \ulTvvgCor ({}^{\mathrm{vvg}}\ul{\psi}M,{}^{\mathrm{vvg}}\ul{\psi}\ul{g}(T)) \xrightarrow{\Gamma_j \circ -} \ulTvvgCor ({}^{\mathrm{vvg}}\ul{\psi}M,T).
\]
It suffices to prove that this is an isomorphism. The injectivity is obvious. We prove the surjectivity. Let $V \in \ulTvvgCor ({}^{\mathrm{vvg}}\ul{\psi}M,T)$ be any integral cycle. It suffices to prove $V \in \ulMCor (M,\ul{g}(T))$.

The assumption that $V$ is in very good position implies
\[
\emptyset = V \cap (|({}^{\mathrm{vvg}}\ul{\psi}M)^-| \times T^\o) = V \cap (M^\o \times |T^-|),
\] 
where the first equality follows from $({}^{\mathrm{vvg}}\ul{\psi}M)^- = \emptyset$. Therefore we have $V \in \Cor (M^\o,T^{\o\o}) = \Cor (M^\o,\ul{g}(T)^\o)$. 

Next we check the left properness of $V$.
Let $\ol{V} \subset \ol{M} \times \ol{T}$ be the closure. Since $V$ is in excellent position, we have 
\[
|({}^{\mathrm{vvg}}\ul{\psi}M)^-| \times \ol{T} \supseteq \ol{V} \cap (\ol{M} \times |T^-|).
\]
Since $({}^{\mathrm{vvg}}\ul{\psi}M)^- = \varnothing$ by definition, the above inclusion shows $\ol{V} \cap (\ol{M} \times |T^-|) = \varnothing$. In other words, $\ol{V} \subset \ol{M} \times (\ol{T} \setminus |T^-|) = \ol{M} \times \ol{\ul{g}(T)}$. This shows the left properness of $V$. 

Finally, we check the modulus condition. Let $\ol{V}^N \to \ol{V}$ be the normalization. Then the assumption that $V \in \ulTvvgCor ({}^{\mathrm{vvg}}\ul{\psi}M,T)$ implies 
\[
M^\infty |_{\ol{V}^N} + T^- |_{\ol{V}^N} \geq T^+ |_{\ol{V}^N} .
\]
By restricting this inequality over $\ol{V} \cap \ol{M} \times (\ol{T} \setminus |T^-|)$, we are done.
\end{proof}

\begin{remark} The identity map $(\P^1,\infty, \emptyset)\to (\P^1,\emptyset,\infty)$ is in very good position, but not in excellent position; its putative image under the functor $\ul{g}$ of Proposition \ref{p3.3} does not satisfy the properness condition.  So this functor does not extend to $ \ulTvgCor $. This motivates the notion of ``excellent position''.

On the other hand, the obvious analogue of Proposition \ref{p3.4} fails if we replace ``very good position'' by ``excellent position''. For example, composing the previous morphism with the morphism $(\P^1,2\infty,\infty)\allowbreak\to (\P^1,\infty,\emptyset)$ with support the identity, we get a morphism in excellent position. Note that the latter is an isomorphism in $\ulTCor$ (see example \ref{ex2.1}).
\end{remark}


\section{Comparisons}\label{s4}
\subsection{Comparison with the categories of Ivorra-Yamazaki}
In \cite[4.5]{iy2}, Ivorra and Yamazaki introduce a quiver (``category without composition'') which is closely related to what we have done so far; they had previously introduced a full subquiver  $\olMCrv$ in \cite[3.2]{iy}. The aim of this subsection is to relate our theory with theirs. By  \cite[Rk 17]{iy2} their quivers fail to yield categories (in the sense that morphisms cannot be naturally composed), but this is in fact a minor issue and we shall modify their definition trivially in order to solve it.

\begin{defn}\label{d3.3} The \emph{Ivorra-Yamazaki category} $\IY$ is defined as follows: 
\begin{description}
\item[Objects]  triples $(X,Y,Z)$, where $X$ is a smooth projective connected $k$-variety and $Y, Z$ are effective divisors on $X$ such that $|Y|\cap |Z| = \emptyset$ and such that $(Y +Z)_\red$ is a simple normal crossing divisor.
\item[Morphisms] a morphism $(X,Y,Z)\to (X',Y',Z')$ is a 
morphism $f:X\to X'$ satisfying 
\begin{enumerate}
\item If $f(X)\subseteq  |Y']$: no condition. (This is how we modify the definition of \cite[4.5]{iy2}.)
\item Otherwise:
\begin{enumerate}
\item $f(X \setminus |Z|) \subseteq X' \setminus |Z'|$
\item $Y\le f^*Y'$
\item $Z - Z_\red \ge f^*(Z' - Z'_\red)$.
\end{enumerate}

\end{enumerate}
\end{description}
\end{defn}

Note that the divisors $f^*Y'$, $f^*(Z'-Z'_\red)$ are well-defined in case (2) (which gives a meaning to Conditions (b) and (c)): for the first it is clear, and for the second it follows from Condition (a). Moreover,

\begin{lemma}\label{rem3.34} In Definition \ref{d3.3}, we always have $f(|Y|) \subseteq |Y'|$, and Condition (2) (a) also holds in Case (1). 
\end{lemma}

\begin{proof} The first point follows from Condition (2) (b) and the second point follows from the hypothesis $|Y'|\cap |Z'| = \emptyset$.
\end{proof}


\begin{prop}\label{IYisCat} $\IY$ is a category. 
\end{prop}

\begin{proof} Let $f:(X,Y,Z) \to (X',Y',Z')$ and $g: (X',Y',Z') \to (X'',Y'',Z'')$ be two morphisms in $\IY$. We check that the composite $gf : X \to X''$ defines a morphism in $\IY$. If $gf(X) \subseteq |Y''|$, we are done. If $gf(X) \not\subseteq |Y''|$, we cannot have $g(X')\subseteq |Y''|$, nor $f(X)\subseteq |Y'|$ by Lemma \ref{rem3.34}. Then all pull-backs are defined, again by Lemma \ref{rem3.34}, and the modulus conditions on $f$ and $g$ immediately imply those on $gf$.
\end{proof}

\begin{remark}\label{r4.1} Thanks to Lemma \ref{rem3.34}, the morphisms verifying Condition (1) of Definition \ref{d3.3} form a $2$-sided ideal of $\IY$. Therefore, one may define a factor category $\ol{\IY}$ by quotienting by this ideal.
\end{remark}

Let $\TSm^\fin= \ulTSm^\fin\cap \TCor$. We want to define a full embedding of $\IY$ into $\TSm^\fin$. For this, we introduce:

\begin{defn} Let 
$\TSm^\fin_{\min}$ denote the full subcategory of $\TSm^\fin$ consisting of triples $(\ol{T},T^+,T^-)$ such that $T^+_\red=F_T := T^+ \times_{\ol{T}} T^-$ (cf. Def. \ref{d2.2}) and $|T^+|\cap |T^--T^+_\red|=\emptyset$, and let $\TSm^\fin_{\min,\ls}$ be the full subcategory of $\TSm^\fin_{\min}$ such that  $\ol{T}$ is smooth, connected and $(T^++T^-)_\red$ is a simple normal crossing divisor. 
\end{defn}

\begin{remark}
An object  $T \in \TSm^\fin$ belongs to $\TSm^\fin_{\min}$ if and only if there exists an effective Cartier divisor $D$ on $\ol{T}$ such that $D \cap T^+ = \emptyset$ and $T^- = T^+_\red + D$.
\end{remark}

\begin{thm}\label{t3.1} The assignments
\begin{align*}
\kappa:\IY&\to \TSm^\fin_{\min,\ls}\\
(X,Y,Z)&\mapsto (X,Z,Y+Z_\red) \\
\kappa':\TSm^\fin_{\min,\ls}&\to\IY\\
(\ol{T},T^+,T^-)&\mapsto (\ol{T},T^--T^+_\red,T^+) 
\end{align*}
define mutually inverse isomorphisms of categories.
\end{thm}

\begin{proof} Clearly, $\kappa$ sends $Ob(\IY)$ into $Ob(\TSm^\fin_{\min,\ls})$ and, conversely, $\kappa'$ is well-defined on objects. This said, it is obvious that they are mutual inverses on objects. It remains to check that  they define functors.

We begin with $\kappa$. Let $(X,Y,Z), (X',Y',Z')\in \IY$ and  $f:X \to X'$ be a morphism. We check that, if $f$ defines a morphism in $\IY$,  it also defines a morphism $\kappa(X,Y,Y) \to \kappa(X',Y',Y')$, i.e., a morphism
\[
(X,Z,Y+Z_\red) \to (X',Z',Y'+Z'_\red)
\]
 in $\TSm^\fin$. By Condition (2) (a) (which always holds by Lemma \ref{rem3.34}), $f(X \setminus Z) \cap |Z'|=\emptyset$, hence $f(X \setminus Z)\not \subseteq |Y'+Z'_\red|$ since $|Y'|\cap |Z'|=\emptyset$. Therefore it suffices to check that (the graph of) $f$ satisfies the modulus condition of \eqref{eq2.1}. If $f(X) \subseteq |Y'|$, then we are done by Lemma \ref{mod-criterion1}. Assume that $f(X) \not\subseteq |Y'|$. Then, we have 
 \begin{align}
Z - Z_\red &\geq f^\ast (Z'-Z'_\red).\label{eq3.14}\\
f^\ast Y' &\geq Y, \label{eq3.15}
\end{align}

The inequalities \eqref{eq3.14} and \eqref{eq3.15} imply
\begin{equation}\label{eq3.12}
(Z - Z_\red) + f^\ast Y' \geq  f^\ast (Z' - Z'_\red) + Y,
\end{equation}
which is equivalent to
\begin{equation}\label{eq3.13}
Z + f^\ast (Y'+Z'_\red) \geq Y + Z_\red + f^\ast Z'.
\end{equation}
This shows that $f$ satisfies the modulus condition, as desired.

To see that $\kappa'$ also defines a functor, it now suffices to show that $f:X\to X'$ defines a morphism in $\IY$ provided either condition of \eqref{eq2.1} is satisfied. Note that Condition (2) (a) always holds, because it means that  $\kappa(f)$ respects the interiors.

If the first condition of \eqref{eq2.1} holds, i.e. $f(X)\subseteq |Y'+Z'_\red|=|Y'|\coprod |Z'_\red|$, then either $f(X)\subseteq |Y'|$ or $f(X)\subseteq |Z'_\red|$ since $X$ is connected. In the first case, we are done. The second case is impossible, as noted above.

If the second condition of \eqref{eq2.1}, i.e. \eqref{eq3.13}, holds, we replace it by the equivalent \eqref{eq3.12}; the conditions $|Y|\cap |Z|=\emptyset$ and $|Y'|\cap |Z'|=\emptyset$ show that  \eqref{eq3.12} implies \eqref{eq3.14} + \eqref{eq3.15}. This yields (b) and (c) in Condition (2) of Definition \ref{d3.3}, concluding the proof.
\end{proof}

\subsection{Functoriality of Ivorra-Yamazaki's mixed Hodge structures with modulus} In \cite[4.4 and 4.5]{iy2}, Ivorra and Yamazaki define representations of (a subquiver of) $\IY$ into their category $\MHSM$ of mixed Hodge structures with modulus. In this subsection, as a slight generalization of their result, we prove that Ivorra-Yamazaki's mixed Hodge structure with modulus is a presheaf on $\IY$. For this, it suffices to extend the definition of their pullbacks to the case of morphisms satisfying Condition (1) of Definition \ref{d3.3}, and to prove compatibility with composition.

Note that for any $(X,Y,Z) \in \IY$, we have a natural inclusion $\Omega^n_{(X,Y,Z)} \subset j_* \Omega^n_{X\setminus Z}$. 
Moreover, for any $f : (X,Y,Z) \to (X',Y',Z')$ in $\IY$, the induced morphism $f^\o : X \setminus Z \to X' \setminus Z'$ defines a pullback map $(f^\o)^*:\Omega^n_{X'\setminus Z'} \to \Omega^n_{X\setminus Z}$.

\begin{lemma}\label{lem:pull-vanish}
Let $f:(X,Y,Z) \to (X',Y',Z')$ be a morphism in $\IY$ such that $f(X) \subset Y'$. Then, for any $p \geq 0$, the composite 
\[
\Omega^n_{(X',Y',Z')} \to j'_* \Omega^n_{X'\setminus Z'} \to f_* j_* \Omega^n_{X\setminus Z}
\]
is the zero map, where $j:X\setminus Z \to X$ and $j':X' \setminus Z' \to X'$ are the open immersions. In other words, we have a commutative diagram 
\[\xymatrix{
\Omega^n_{(X',Y',Z')} \ar[r]^0 \ar[d] & f_* \Omega^n_{(X,Y,Z)} \ar[d] \\
\Omega^n_{X'\setminus Z'} \ar[r]^{(f^\o)^*} & f_* j_* \Omega^n_{X\setminus Z}
}\]
where the vertical arrows are the inclusions. 
\end{lemma}

\begin{proof} By adjunction, it suffices to show that the composite
\[
f^*\Omega^n_{(X',Y',Z')} \to f^*j'_* \Omega^n_{X'\setminus Z'} \to  j_* \Omega^n_{X\setminus Z}
\]
is the zero map.
Since the problem is Zariski local on $X'$, we may assume that the simple normal crossing divisor $|Y'|$ is defined by a local parameter $s=s_1^{m_1}\cdots s_r^{m_r}$, where $s_1,\dots,s_r$ are a system of regular sequences and $m_i$ are positive integers. 

Since $X$ is integral and $f(X) \subset Y'$, the morphism $f$ factors through an irreducible component of $|Y'|$ with its reduced structure, which we denote by $Y'_1$.
By changing the numbering if necessary, we may assume that $Y'_1$ is defined by $s_1=0$.
Since we have $\Omega^n_{(X',Y',Z')} \subset \Omega^n_{(X',Y'_1,Z')} \subset j'_* \Omega^n_{(X'\setminus Z',Y'_1,\emptyset )}$ by definition, we may assume that $Z'=\varnothing$ and $Y'$ is an integral smooth divisor whose local parameter is $s=s_1$.

When $Z'=\varnothing$ and $Y'$ is an integral smooth divisor parametrized by $s$, we have 
\begin{align*}
\Omega^n(X',Y',\varnothing) &= \Omega^n(\log Y'_\red) \otimes_{\mathcal{O}_{X'}} \mathcal{O}_{X'}(-Y') \\
&= \left(\bigwedge^n \Omega^1(\log Y'_\red)  \right) \otimes_{\mathcal{O}_{X'}} \mathcal{O}_{X'}(-Y').
\end{align*}

Note that $\Omega^1(\log Y'_\red)=\Omega^1_{X'}\oplus \sO_{X'}\mathrm{dlog} s$. It follows that 
$\bigwedge^n \Omega^1(\log Y'_\red)$ is a free $\mathcal{O}_{X'}$-module generated by the exterior products of $\Omega^1_{X'}$ and $\mathrm{dlog}s$, and moreover, in each exterior products, $\mathrm{dlog}s$ appears at most once. 
Therefore, the $\mathcal{O}_{X'}$-module $ \left(\bigwedge^n \Omega^1(\log Y'_\red)  \right) \otimes_{\mathcal{O}_{X'}} \mathcal{O}_{X'}(-Y') = s \cdot \left(\bigwedge^n \Omega^1(\log Y'_\red)  \right)$ is generated $s \cdot \Omega^1_{X'}$ and $s\mathrm{dlog} s = s \cdot (ds/s) = ds$.
Since we have $f^*s = 0=f^*ds$ by the factorization $X \to Y' \subset X'$, we see that the pullback by $f$ of the above generators vanish. 
\end{proof}

By Lemma \ref{lem:pull-vanish}, one sees that there exists a natural pullback $f^*:H^{n,k}(X',Y',Z’) \to H^{n,k}(X,Y,Z)$ for any morphism $f$ in $\IY$ extending the one of \cite{iy2}, and if moreover $f(X) \subset |Y'|$, then $f^*=0$.
This easily implies that $f^*$ is compatible with the composition in $\IY$, and hence $H^{n,k}(X,Y,Z)$ is a presheaf on $\IY$.
A similar argument shows that $H_{\mathrm{add}}$ is also a presheaf (and $H_{\mathrm{inf}}$ is a presheaf by a different and easier reason). If we understood the definition of $\MHSM$ correctly, this defines a functor $\IY^\op\to \MHSM$, which factors through the category $\ol{\IY}$ of Remark \ref{r4.1}.

\begin{remarks} Let $d_1 \TSm^\fin$ be the full subcategory of  $\TSm^\fin$ formed of those $T$ such that $\dim \ol{T}=1$.
\begin{enumerate}
\item  The assignment 
\[(\ol{T},T^+,T^-)\mapsto (\ol{T},T^++T^+_\red,T^-+T^+_\red)\]
defines an endofunctor $\sigma$ of $d_1 \TSm^\fin$, provided with a natural isomorphism $\sigma\Rightarrow \id$. This endofunctor sends $d_1\TphiSm^\fin$ into $d_1\TSm^\fin_{\min}$, hence realises the former category as the full subcategory of $\olMCrv$ consisting of triples $(X,Y,Z)$ such that $Z$ is ``everywhere non reduced''. 
\item Any object of $d_1\TSm^\fin$ is isomorphic to a unique object $(\ol{T},T^+,T^-)$ such that $\ol{T}$ is normal and $F_T$ is reduced. Hence this category is too big to describe $\olMCrv$.
\item In \cite{iy}, Ivorra and Yamazaki use  $\olMCrv$ to perform a construction \`a la Nori, which they identify with Laumon's $1$-motives when $k$ is a number field. 
\item The category $\MHSM$ carries a perfect duality, but no tensor structure. Can one modify it so that a tensor structure becomes available and the representations of \cite{iy2} extend via Theorem \ref{t3.1} to a graded $\otimes$-functor from $\TSm^\fin$?
\end{enumerate}
\end{remarks}

Here are two ``cross-fertilisation'' ideas between $\TCor$ and $\IY$, trying to answer the last question:

\begin{description}
\item[Duality] the category $\IY$ carries a duality compatible with the Hodge realisation \cite[Th. 18]{iy2}: it simply carries $(X,Y,Z)$ to $(X,Z,Y)$. Translated through the functors $\kappa$ and $\kappa'$ of Theorem \ref{t3.1}, it yields the following in $\TSm^\fin_{\min,\ls}$:
\[(\ol{T},T^+,T^-)\mapsto (\ol{T},T^--T^+_\red,T^++(T^--T^+_\red)_\red).\]

(One should also twist by $\dim \ol{T}$, but twists would have to be understood later for modulus triples.)
\item[Tensor product] towards answering Question (3) above). The category $\IY$ lacks a tensor structure, but $\TSm$ has one by Proposition \ref{p2.1}. We can try to transport it to $\IY$ as above getting the same formula, namely
\begin{multline*}
(X,Y,Z)\otimes (X',Y',Z')=\\
(X\times X',Y\times X'+X\times Y',Z\times X'+X\times Z').
\end{multline*}

However, this tensor product does not have disjoint supports, so it does not define an object of $\IY$.
\end{description}

\subsection{Comparison with Binda's modulus data} Let us ``recall'' from \cite[Def. 2.11]{binda} the definition of the category $\olMSm$ of modulus data:

\begin{description}
\item[Objects] triples $M = (\ol{M};\partial M,D_M)$, where $\ol{M} \in \Sm$ is a smooth $k$-scheme, $\partial M$ is a {\color{blue} reduced} strict normal crossing divisor on $\ol{M}$ (possibly empty), $D_M$ is an effective Cartier divisor on $\ol{M}$ (the case $D_M = \emptyset$ is allowed), and the total divisor $(D_M)_\red + \partial M$ is a strict normal crossing divisor on $\ol{M}$\footnote{ According to loc. cit., Rem. 2.12, this latter condition is not strictly necessary.}.
\item[Morphisms] given $M_1, M_2\in \olMSm$, a $k$-morphism $f : \ol{M_1} \to \ol{M_2}$ is \emph{admissible} (i.e. defines a morphism $M_1\to M_2$) if it satisfies the following conditions.
\begin{enumerate}
\item[i)] For every irreducible component $\partial M_{2,l}$ of $\partial M_2$ we have $|f^*(\partial M_{2,l})| \subseteq |\partial M_1|$.
\item[ii)]  If $f(\ol{M}_1)\subseteq  |D_{M_2}]$: no further condition.
\item[iii)] Otherwise: $f^*(D_{M_2} )\allowbreak \ge D_{M_1}$ as Weil divisors on $\ol{M}_1$.
\end{enumerate}
\end{description}

Here are some comments on this reminder. First, the fact that $\partial M$ is reduced is implicit in \cite{binda}, as only its support is used. In the definition of morphisms, Case ii) is missing from \cite{binda} but seems necessary to ensure their composability, as in $\IY$. 
Thanks to this, up to switching $\partial M$ and $D_M$, Condition i) on morphisms is identical to Condition (2) (a) of Definition \ref{d3.3}, Part (c) of the latter being empty, and Condition i) is identical to Conditions (1) + (2) (a) of Definition \ref{d3.3}. This formulation of Condition ii) is necessary here, because the second point of Lemma \ref{rem3.34} depends on the disjointness of supports in $\IY$ which is not assumed in \cite{binda}.

This said, let $\TSm^\fin_\man$ denote the full subcategory of $\TSm^\fin$ consisting of triples $(\ol{T},T^+,T^-)$ such that $T^+$ is reduced and  $T^--T^+$ is  an effective Cartier divisor, and let $\TSm^\fin_{\man,\ls}$ be the full subcategory of $\TSm^\fin_\man$ such that  $\ol{T}$ is smooth, connected and $T^-_\red$ is a simple normal crossing divisor. Then the proof of Theorem \ref{t3.1} carries over to give two mutually inverse isomorphisms of categories
\begin{align*}
\kappa:\olMSm&\to \TSm^\fin_{\man,\ls}\\
(\ol{M};\partial M,D_M)&\mapsto (\ol{M},\partial M,D_M+\partial M) \\
\kappa':\TSm^\fin_{\man,\ls}&\to\olMSm\\
(\ol{T},T^+,T^-)&\mapsto (\ol{T},T^+, T^--T^+). 
\end{align*}




\appendix

\section{Non-effective modulus}

Morally, we would like to think of the modulus triple $T=(\ol{T},T^+,T^-)$ as a ``model'' of a pair $(\ol{T},T^+-T^-)$, with $T^+-T^-$ a Cartier divisor which is not necessarily effective. 
Actually, we can define the category of modulus pairs with ``non-effective modulus'', and we can embed it fully faithfully into the category of modulus triples. 

\begin{definition}
We say that a modulus triple $T$ is \emph{saturated} if $|T^+ - T^-| = |T^+|$.
Let $\ulTsatCor$ be the full subcategory of $\ulTCor$ consisting of saturated modulus triples. 
\end{definition}



\begin{definition}
An \emph{ne-modulus pair} is a pair $\sX=(\ol{X},X^\infty)$ such that 
\begin{enumerate}
\item $\ol{X}$ is a separated \emph{normal} $k$-scheme of finite type,
\item $X^\infty$ is a Cartier divisor (not necessarily effective!),
\item $\sX^{\o}:=\ol{X} \setminus |X^\infty|$ is $k$-smooth.
\end{enumerate}
\end{definition}

Recall that, by definition, $|X^\infty|$ is the (closed) complement in $\ol{X}$ of the locus where $X^\infty$ is $0$.

\begin{remark}
The ``ne'' stands for ``non-effective''. Note that any modulus pair whose total space is normal is an ne-modulus pair. 
\end{remark}

\begin{definition}
For ne-modulus pairs $\sX,\sY$, define $\ulNCor (\sX,\sY)$ to be the free abelian group generated by those elementary correspondences $V \in \Cor (\sX^\o , \sY^\o)$ such that the following condition holds:
let $\ol{V} \subset \ol{X} \times \ol{Y}$ be the closure of $V$ and $\ol{V}^N \to \ol{V}$ be its normalization. Let $p:\ol{V}^N \to \ol{X}$ and $q:\ol{V}^N \to \ol{Y}$ be the natural projections.
Then 
\begin{enumerate}
\item $p$ is proper.
\item $p^*X^\infty \geq q^*Y^\infty$.
\end{enumerate}
We call an element of $\ulNCor (\sX,\sY)$ a \emph{modulus correspondence} from $\sX$ to $\sY$.
\end{definition}

(These are the same conditions as in $\ulMCor$.)

\begin{lemma}\label{lA.1}
The composition of finite correspondences induce composition of modulus correspondences. In particular, we obtain a category $\ulNCor$ whose objects are ne-modulus pairs and whose morphisms are modulus correspondences. Moreover, there exists a fully faithful functor $\ulMCor \to \ulNCor$.
\end{lemma}

\begin{proof}
The proof of the well-definedness the composition is verbatim the same as the one for the composition in $\ulMCor$ \cite[Prop. 1.2.4 and 1.2.7]{motmodI}.
The functor $\ulMCor \to \ulNCor$ is given by $\sX=(\ol{X},X^\infty) \mapsto \sX^N :=(\ol{X}^N,p^*X^\infty)$, where $p:\ol{X}^N \to \ol{X}$ is the normalization. 
Noting that $\sX^N \cong \sX$ in $\ulMCor$, the full faithfulness is obvious by definition of $\ulNCor$.
\end{proof}

The support of the pullback of a (non-effective) Cartier divisor $D$ may be strictly smaller than the pullback of the support of $D$. However, we have the following lemma. 

\begin{lemma}\label{lem:pb-supp}
Let $f:Y \to X$ be a proper surjective morphism of integral schemes such that the canonical morphism $\mathcal{O}^*_X \to f_* \mathcal{O}^*_Y$ is an isomorphism. 
Let $D$ be a Cartier divisor on $X$. 
Then we have $|f^*D| = f^{{-1}} (|D|)$ as sets. 
\end{lemma}

\begin{proof}
We have an obvious inclusion $|f^*D| \subset f^{{-1}} (|D|)$. 
We have to prove $|f^*D| \supset f^{{-1}} (|D|)$, i.e., $f(|f^*D|) \supset |D|$.
Noting that $f(|f^*D|) \subset X$ is a closed subset by the properness of $f$, this is equivalent to that $D$ is trivial on the open subset $X \setminus f(|f^*D|)$.
Notice that $f^*D$ is trivial on $Y \setminus f^{-1} (f(|f^*D|)) \subset Y \setminus |f^*D|$.
Therefore, replacing $X$ by $X \setminus f(|f^*D|)$, we are reduced to the case that $f^*D$ is trivial. 

Assume that $f^*D$ is trivial. What we have to prove is that $D$ is also trivial. 
Consider the commutative diagram
\[\xymatrix{
0 \ar[r] & \mathcal{O}^*_X \ar[r] \ar[d]_a & \mathcal{K}^*_X \ar[r] \ar[d]_b & \mathcal{K}^*_X / \mathcal{O}^*_X \ar[r] \ar[d] _c& 0 \\
0 \ar[r] & f_* \mathcal{O}^*_Y \ar[r] & f_* \mathcal{K}^*_Y \ar[r] & f_* (\mathcal{K}^*_Y / \mathcal{O}^*_Y) & 
}\]
where $a$ is an isomorphism by assumption and $b$ is an injection which exists by the dominance of $f$.
Note that the horizontal sequences are exact since $f_*$ is left exact. 
Then, the snake lemma shows that $c$ is injective, which means that a Cartier divisor on $X$ is trivial if its pullback on $Y$ is trivial. 
\end{proof}

\begin{cor}\label{cor:pb-supp}
Let $f:Y \to X$ be a proper surjective morphism of integral schemes with $X$ normal. Then, for any Cartier divisor $D$ on $X$, we have $|f^*D| = f^{{-1}} (|D|)$ as sets. 
\end{cor}

\begin{proof}
By the properness of $f$, we know that $f_*\mathcal{O}_Y$ is a finite $\mathcal{O}_X$-algebra. Then the normality of $X$ shows $f_*\mathcal{O}_Y = \mathcal{O}_X$.
In particular, we have $f_*\mathcal{O}^*_Y = \mathcal{O}^*_X$. 
Then the assertion follows from Lemma \ref{lem:pb-supp}.
\end{proof}

\begin{prop}\label{pA.1}
There exists a fully faithful functor $\ulNCor \to \ulTCor$ whose essential image is $\ulTsatCor$.
In other words, there exists an equivalence of categories $\ulNCor \xrightarrow{\sim} \ulTsatCor$.
\end{prop}

\begin{proof}
First we construct a functor $i:\ulNCor \to \ulTCor$ as follows. 
Let $\sX = (\ol{X},X^\infty)$ be an ne-modulus pair.
Let $I \subset \mathcal{O}_{\ol{X}}$ be the ideal of denominators of the Cartier divisor $X^\infty$ \cite[p. 37]{fulton}, and let $F \subset \ol{X}$ be the closed subscheme defined by $I$. 
Let $p:\ol{X}' \to \ol{X}$ be the blow-up of $\ol{X}$ along $F$, and let $E:=F \times_{\ol{X}} \ol{X}'$ be the exceptional divisor. 
Set
\[
i(\sX) := (\ol{X}',p^*X^\infty + 2E,2E).
\]

\begin{claim}
$i(\sX)^\o = \sX^\o$. Equivalently, $|p^*X^\infty+2E| = p^{-1}(|X^\infty|)$
\end{claim}

\begin{proof}[Proof of Claim]
The equivalence is by Corollary \ref{cor:pb-supp}. We prove the latter assertion. 
Since both sides coincide on $\ol{X}' \setminus E$ by Corollary \ref{cor:pb-supp}, it suffices to check 
\[
|p^*X^\infty+2E| \cap |E| = p^{-1}(|X^\infty|) \cap |E|.
\]
The right hand side equals $|X^\infty|$ since $|E| = p^{-1}(F)$ and $F \subset |X^\infty|$ (as sets). 
The left hand side also equals $|E|$ since $p^*X^\infty+E$ is effective and hence $p^*X^\infty+2E \supset E$ (as closed subschemes). 
\end{proof}

Let $\alpha \in \ulNCor (\sX,\sY)$ be an integral modulus correspondence. 
Notice that $\ulNCor (\sX,\sY) \subset \Cor (\sX^\o , \sY^\o) = \Cor (i(\sX)^\o , i(\sY)^\o)$ by the above claim. Then one easily checks that $\alpha$ belongs to $\ulTCor (i(\sX),i(\sY))$ and that $\ulNCor (\sX,\sY)= \ulTCor (i(\sX),i(\sY))$ in $\Cor (\sX^\o , \sY^\o)$. 
Moreover, by the poof of the above claim, we have 
\[
|i(\sX)^+ - i(\sX)^-| = |(p^*X^\infty + 2E) - 2E| = |p^*X^\infty| = |i(\sX)^+|,
\]
which shows $i(\sX) \in \ulTsatCor$. 

It remains to show that $i:\ulNCor \to \ulTsatCor$ is essentially surjective. 
Let $T$ be a saturated modulus triple, and set $\sX:=(\ol{T},T^+ - T^-)$. 
Let $I$ be the ideal of denominators of $T^+ - T^-$ and let $F \subset \ol{T}$ be the closed subscheme defined by $I$. Note that $F \subset |T^+ - T^-| = |T^+|$, where the equality holds since $T$ is saturated by assumption. 
Let $p:\ol{T}' \to \ol{T}$ be the blow-up along $F$ and let $E=p^{-1}(F)$ be the exceptional divisor. Then, by definition we have 
\[
i(\sX) = (\ol{T}', p^*(T^+-T^-) + 2E, 2E).
\]
Then $p$ induces a morphism of triples $i(\sX) \to T$, which is the identity on the interiors (since $i(\sX)^\o = \sX^\o = \ol{T} \setminus |T^+-T^-| = \ol{T} \setminus |T^+| = T^\o$).
This identity also defines an inverse morphism $T \to i(\sX)$. Indeed, the left properness is guaranteed by the properness of $p$, and the modulus condition is obviously satisfied noting $i(\sX)^+ - i(\sX)^- = p^*(T^+-T^-) + 2E - 2E = p^*(T^+-T^-)$. 
\end{proof}

\end{document}